\documentclass[final]{siamltex}

\usepackage{graphicx}
\usepackage{url}
\usepackage{epstopdf}

\usepackage{amsmath,amssymb}
\usepackage{subfig}
\usepackage{color}
\usepackage{algorithm2e}

\graphicspath{{./}{./figs/}}

  \def\clap#1{\hbox to 0pt{\hss#1\hss}}

\usepackage{bm}
\providecommand{\mat}[1]{\bm{#1}}%
\renewcommand{\vec}[1]{\mathbf{#1}}

\providecommand{\mA}{\ensuremath{\mat{A}}}
\providecommand{\mB}{\ensuremath{\mat{B}}}
\providecommand{\mC}{\ensuremath{\mat{C}}}

\providecommand{\mE}{\ensuremath{\mat{E}}}

\providecommand{\mG}{\ensuremath{\mat{G}}}
\providecommand{\mH}{\ensuremath{\mat{H}}}

\providecommand{\mU}{\ensuremath{\mat{U}}}
\providecommand{\mV}{\ensuremath{\mat{V}}}
\providecommand{\mW}{\ensuremath{\mat{W}}}
\providecommand{\mX}{\ensuremath{\mat{X}}}

\providecommand{\mSigma}{\ensuremath{\mat{\Sigma}}}

\providecommand{\vg}{\ensuremath{\vec{g}}}

\providecommand{\vs}{\ensuremath{\vec{s}}}
\providecommand{\vt}{\ensuremath{\vec{t}}}

\providecommand{\vv}{\ensuremath{\vec{v}}}
\providecommand{\vw}{\ensuremath{\vec{w}}}
\providecommand{\vx}{\ensuremath{\vec{x}}}
\providecommand{\vy}{\ensuremath{\vec{y}}}
\providecommand{\vz}{\ensuremath{\vec{z}}}

\newcommand{\hmB}{\hat{\mB}}
\newcommand{\hmC}{\hat{\mC}}
\newcommand{\hmG}{\hat{\mG}}
\newcommand{\hSigma}{\hat{\Sigma}}
\newcommand{\hLambda}{\hat{\Lambda}}
\newcommand{\hTheta}{\hat{\Theta}}
\newcommand{\hmU}{\hat{\mU}}
\newcommand{\hmV}{\hat{\mV}}
\newcommand{\hmW}{\hat{\mW}}
\newcommand{\hlambda}{\hat{\lambda}}
\newcommand{\htheta}{\hat{\theta}}

\newcommand{\sX}{\mathcal{X}}

\newcommand{\Exp}[1]{\mathbb{E}\left[#1\right]}

\newcommand{\Prob}[1]{\mathbb{P}\left\{#1\right\}}
\newcommand{\bmat}[1]{\begin{bmatrix}#1\end{bmatrix}}

\newcommand{\trace}[1]{\operatorname{trace}\left(#1\right)}
\newcommand{\gf}{\nabla_\vx f}
\newcommand{\intrho}[1]{\int #1 \,\rho\,d\vx}
\newcommand{\lmax}{\lambda_{\text{max}}}
\newcommand{\dist}[2]{\operatorname{dist}(\operatorname{ran}(#1),\,\operatorname{ran}(#2))}

\title{COMPUTING ACTIVE SUBSPACES WITH MONTE CARLO}

\author{
Paul G.~Constantine\thanks{Department of Applied Mathematics and Statistics, Colorado School of Mines, Golden, CO 80401 (\texttt{paul.constantine@mines.edu}).}
\and
David F.~Gleich\thanks{Department of Computer Science, Purdue University, West Lafayette, IN (\texttt{dgleich@purdue.edu})}
}

\begin{document}
\maketitle

\begin{abstract}
Active subspaces can effectively reduce the dimension of high-dimensional parameter studies enabling otherwise infeasible experiments with expensive simulations. The key components of active subspace methods are the eigenvectors of a symmetric, positive semidefinite matrix whose elements are the average products of partial derivatives of the simulation's input/output map. We study a Monte Carlo method for approximating the eigenpairs of this matrix. We offer both theoretical results based on recent non-asymptotic random matrix theory and a practical approach based on the bootstrap. We extend the analysis to the case when the gradients are approximated, for example, with finite differences. Our goal is to provide guidance for two questions that arise in active subspaces: (i) How many gradient samples does one need to accurately approximate the eigenvalues and subspaces? (ii) What can be said about the accuracy of the estimated subspace, both theoretically and practically? We test the approach on both simple quadratic functions where the active subspace is known and a parameterized PDE with 100 variables characterizing the coefficients of the differential operator. 
\end{abstract}


\begin{AMS}\end{AMS}

\pagestyle{myheadings}
\thispagestyle{plain}
\markboth{PAUL G.~CONSTANTINE AND DAVID F.~GLEICH}{COMPUTING ACTIVE SUBSPACES WITH MONTE CARLO}

\section{Introduction}
\label{sec:intro}

Engineering models typically contain several input parameters that must be specified to produce a set of model outputs that contains one or more quantities of interest. The engineer's goal is to characterize the behavior of the quantities of interest as functions of the model's inputs. However, parameter studies---such as optimization and uncertainty quantification---are challenging when the number of inputs is large and the model involves an expensive computer simulation. In such cases, the engineer may analyze the output's sensitivity with respect to inputs to identify a subset of inputs whose variation changes the outputs the most~\cite{saltelli2008global}. In the best case, she can then limit parameter studies to key parameters and thus reduce the \emph{dimension} of the parameter study. This approach is appropriate when varying important parameters changes the outputs much more than varying the unimportant parameters. However, a model's output may depend on all the parameters through certain linear combinations, which generalizes seeking key parameters to seeking key directions in the parameter space. The \emph{active subspace} identifies important directions in the model's input space with respect to a particular quantity of interest; perturbing the inputs along these important directions changes the quantity of interest more, on average, than perturbing the inputs in orthogonal directions~\cite{constantine2014active}. For parameter studies whose work depends exponentially on the number of parameters---e.g., integration or response surface construction---the active subspace-enabled dimension reduction can permit otherwise infeasible studies.  

The active subspace is defined by a set of eigenvectors corresponding to large eigenvalues of the average outer product of the gradient with itself. These eigenpairs are properties of the map between model inputs and outputs, like Fourier coefficients or the Lipschitz constant. To determine if a function admits a low-dimensional active subspace---and thus reduce the dimension of the parameter studies---one must estimate these eigenpairs. This estimation is problematic because the elements of the matrix defining the eigenpairs are themselves high-dimensional integrals. Most deterministic numerical integration rules are impractical beyond a handful of variables, especially if the integrand is costly to evaluate. We therefore focus on a Monte Carlo approach to approximate the eigenpairs, where we take advantage of recent theoretical results that bound the number of samples needed to approximate the spectrum of sums of random matrices. Monte Carlo is attractive because it makes few restrictions on the function defining the quantity of interest. Under additional assumptions, one may be able to outperform Monte Carlo with specialized integration rules for integrands that depend on many variables, e.g., with sparse grids~\cite{Bungartz2004} or quasi-Monte Carlo~\cite{Caflisch1998}. 


In what follows, we analyze a Monte Carlo method for estimating the eigenpairs that uses independent samples of the function's gradient. After formally defining the active subspace in Section \ref{sec:asm}, we employ results from Tropp~\cite{tropp2012user} and Gittens and Tropp~\cite{gittens2011tail} to bound the probability that the estimated eigenvalues deviate from the true eigenvalues, which yields lower bounds on the number of samples needed for accurate estimation. We extend these results to the case where samples are approximate gradients (e.g., finite difference approximations). In Section \ref{sec:prac} we discuss a practical bootstrap approach to study the variability in the estimated eigenvalues, and we demonstrate these procedures numerically in Section \ref{sec:experiments}.

\paragraph{Notation}
We use bold lower case letters to denote vectors and bold upper case letters to denote matrices. Finite sample estimates are denoted with hats, e.g., $\hmC\approx\mC$. The functional $\lambda_k(\cdot)$ denotes the $k$th eigenvalue of its argument, ordered from algebraically largest to smallest; all matrices are symmetric, so the ordering is meaningful. A $\lambda$ on its own is an eigenvalue. Norms of vectors and matrices are 2-norms; the matrix 2-norm is the operator-induced norm given by the largest singular value. The partial ordering operator $\preceq$ is defined as follows: $\mA\preceq \mB$ means that $\mB-\mA$ is positive semidefinite.

\section{Active subspaces}
\label{sec:asm}

We represent the map from simulation inputs to the scalar-valued quantity of interest by a function $f:\sX\rightarrow\mathbb{R}$, where $\sX\subseteq\mathbb{R}^m$, with $m>1$, represents the set of simulation inputs, which we assume is centered at the origin and scaled so that each component of $\vx\in\sX$ has the same range. Let $\mathbb{R}^m$ be equipped with a weight function $\rho:\mathbb{R}^m\rightarrow\mathbb{R}_+$ that is bounded, strictly positive on the domain $\sX$, and zero outside of $\sX$. We also assume that $\rho$ is both separable and normalized to integrate to 1. In the context of uncertainty quantification, this weight function represents a given probability density function on the inputs; examples in this context include Gaussian, uniform, or data-conditioned Bayesian posterior density functions. We assume $f$ is differentiable and absolutely continuous, and we denote the gradient $\nabla_\vx f(\vx) = [\partial f/\partial x_1,\dots,\partial f/\partial x_m ]^T$ oriented as a column vector. 

We are interested in the following matrix, denoted $\mC$ and defined as
\begin{equation}
\label{eq:C}
\mC \;=\; \intrho{(\nabla_\vx f)(\nabla_\vx f)^T}.
\end{equation}
Samarov studied this matrix as one of several \emph{average derivative functionals} in the context of regression, where $f$ is the regression function~\cite{samarov1993exploring}. The matrix $\mC$ is symmetric and positive semi-definite, so it has a real eigenvalue decomposition
\begin{equation}
\label{eq:eigdecomp}
\mC \;=\; \mW\Lambda\mW^T, 
\qquad 
\Lambda \;=\; \mathrm{diag}(\lambda_1,\dots,\lambda_m),
\qquad
\lambda_1\geq\cdots \geq\lambda_m\geq 0,
\end{equation}
where $\mW$ is the orthogonal matrix of eigenvectors. Partition the eigenpairs,
\begin{equation}
\label{eq:splitting}
\mW = \bmat{\mW_1 & \mW_2},\qquad \Lambda = \bmat{\Lambda_1 & \\ & \Lambda_2},
\end{equation}
where $\mW_1$ contains the first $n<m$ eigenvectors, and $\Lambda_1$ contains the $n$ largest eigenvalues. The eigenvectors define new coordinates
\begin{equation}
\label{avs}
\vy=\mW_1^T\vx,\qquad\vz=\mW_2^T\vx.
\end{equation}
We call the column space of $\mW_1$ the \emph{active subspace} and the corresponding $\vy$ the \emph{active variables}. Similarly, $\mW_2$ defines the \emph{inactive subspace} with corresponding \emph{inactive variables} $\vz$. The following two lemmas justify this characterization; these are proved in~\cite{constantine2014active}.

\begin{lemma}
\label{lem:avgsqgrad}
The mean-squared directional derivative of $f$ with respect to the eigenvector $\vw_i$ is equal to the corresponding eigenvalue,
\begin{equation} 
\intrho{\big((\nabla_{\vx} f)^T\vw_i\big)^2} \;=\; \vw_i^T \mC \vw_i \;=\; \lambda_i.
\end{equation}
\end{lemma}

\begin{lemma}
\label{lem:grad}
The mean-squared gradients of $f$ with respect to the coordinates $\vy$ and $\vz$ satisfy
\begin{equation}
\begin{aligned}
\intrho{(\nabla_{\vy} f)^T(\nabla_{\vy} f)} &= \trace{\mW_1^T \mC \mW_1} = \lambda_1+\cdots+\lambda_n, \\
\intrho{(\nabla_{\vz} f)^T(\nabla_{\vz} f)} &= \trace{\mW_2^T \mC \mW_2} = \lambda_{n+1}+\cdots+\lambda_m.
\end{aligned}
\end{equation}
\end{lemma}
The eigenvalues $\Lambda$ and eigenvectors $\mW$ are properties of $f$. If the $m-n$ trailing eigenvalues $\Lambda_2$ are exactly zero, then $f$ is constant along the directions corresponding to $\mW_2$. If $\Lambda_2$ is not exactly zero but significantly smaller than $\Lambda_1$, then $f$ changes less, on average, in response to small changes in $\vz$ than small changes in $\vy$. If $f$ admits such a property, we would like to discover and exploit it in parameter studies by focusing on the variables $\vy$. In other words, we can reduce the dimension of the parameter studies from $m$ to $n<m$. 

Two special cases illustrate the active subspace. The first class of functions are index models that have the form $f(\vx)=h(\mA^T\vx)$, where $\mA\in\mathbb{R}^{m\times k}$ and $h:\mathbb{R}^k\rightarrow\mathbb{R}$. In this case, $\mC$ has rank at most $k$, and the active subspace is a subspace of the range of $\mA$. If $k=1$, then the one-dimesional active subspace can be discovered with a single evaluation of $\nabla_\vx f$ at any $\vx\in\sX$ such that the derivative $h'(\mA^T\vx)$ is not zero.
The second special case is a function of the form $f(\vx) = h(\vx^T\mH\vx)/2$, where $h:\mathbb{R}\rightarrow\mathbb{R}$, and $\mH$ is a symmetric $m\times m$ matrix. In this case
\begin{equation}
\mC \;=\; \mH \left(\intrho{(h')^2\,\vx\vx^T}\right)\mH^T, 
\end{equation}
where $h'=h'(\vx^T\mH\vx)$ is the derivative of $h$. This implies that the null space of $\mC$ is the null space of $\mH$ provided that $h'$ is non-degenerate. We study the example where $h(t) = t$ in Section \ref{sec:experiments}. 

If we can estimate $\Lambda$ and $\mW$ from \eqref{eq:eigdecomp}, then we can approximate $f(\vx)$ with a model of the form
\begin{equation}
f(\vx) \;\approx\; g(\mW_1^T\vx),
\end{equation}
where $g:\mathbb{R}^n\rightarrow\mathbb{R}$ is an appropriately constructed map. In~\cite{constantine2014active}, we derive error bounds for this approximation with a particular choice of $g$. We extend those error bounds to the case when $\mW_1$ is estimated with some error. The main goal of this paper is to study the error in $\mW_1$ when $\mC$ is estimated with Monte Carlo. 

\subsection{Related literature}
The idea of studying the eigenpairs of the average outer product of the gradients arose in statistics as \emph{average derivative functionals}~\cite{samarov1993exploring,hristache2001} for exploring structure in regression functions. In contrast to our $f(\vx)$, the regression function is generally unknown; to estimate the gradients of the unknown regression function, one can first fit a kernel-based model to a set of predictor/response pairs and then compute gradients from the approximation~\cite{Xia2007,fukumizu2014gradient}. In our case, the function is a map between the inputs and outputs of an engineering simulation; there is no random noise as in the regression problem. The set up in Russi's Ph.D.~thesis~\cite{Russi2010} is closer to ours. He applies the methods to physical simulations of chemical kinetics; this work is where we encountered the term \emph{active subspace}. Recent work in approximation theory by Fornasier, Schnass, and Vybiral~\cite{Fornasier2012} attempts to discover the parameters of the active subspace solely through queries of the function; guarantees on reconstruction follow from compressed sensing results under the assumption that $f$ is an index model. 

If the matrix $\mC$ were given as an input matrix, or if we could easily compute matrix-vector products with $\mC$, then we could employ recent procedures for randomized low-rank approximation to estimate the desired eigenpairs~\cite{Halko-2011-randomness,Gittens-2013-nystrom}---assuming $\mC$ is well approximated by a low-rank matrix, which is often the case in practice. Unfortunately, we do not have easy access to the elements of $\mC$; estimating its eigenpairs requires estimating its elements. There may be fruitful relationships with
low-rank approximation of quasimatrices and 
cmatrices~\cite{Townsend-2014-continuous} that are worth exploring.

\section{Computing active subspaces}
\label{sec:comp}


If drawing independent samples from the density $\rho$ is cheap and simple, then a straightforward and easy-to-implement random sampling method to approximate the eigenvalues $\Lambda$ and eigenvectors $\mW$ proceeds as follows.
\begin{enumerate}
\item Draw $N$ samples $\vx_j$ independently from the measure $\rho$. 
\item For each $\vx_j$, compute $\nabla_\vx f_j = \nabla_\vx f(\vx_j)$. 
\item Approximate 
\begin{equation}
\label{eq:compC}
\mC \;\approx\;
\hmC \;=\;
\frac{1}{N} \sum_{j=1}^N (\nabla_\vx f_j)(\nabla_\vx f_j)^T.
\end{equation}
\item Compute the eigendecomposition $\hmC=\hmW\hLambda\hmW^T$.
\end{enumerate}
The last step is equivalent to computing the full SVD of the matrix 
\begin{equation}
\label{eq:gmat}
\hmB \;=\; \frac{1}{\sqrt{N}} \bmat{\nabla_\vx f_1 & \cdots & \nabla_\vx f_N}
\;=\;
\hmW\hSigma\hmV^T,
\end{equation}
where standard manipulations show that the $\hLambda=\hSigma\hSigma^T$, and the left singular vectors are the desired eigenvectors. The SVD perspective was developed by Russi~\cite{Russi2010} as the method to discover the active subspace. This SVD perspective also calls to mind randomized methods for subsampling the columns of $\hmB$, where $N\gg m$~\cite{Gittens-2013-nystrom}. If it were possible to evaluate the importance of a column of $\hmB$ without explicitly computing $\nabla_{\vx} f(\vx)$, then such ideas might prove useful.

For many simulations, the number $m$ of input parameters is small enough (e.g., tens to thousands) that computing the full eigendecomposition \eqref{eq:compC} or singular value decomposition \eqref{eq:gmat} is negligible compared to the cost of computing the gradient $N$ times; we consider this to be our case of interest. We are therefore concerned with understanding the number of gradient samples needed so that the estimates $\hLambda$ and $\hmW$ are close to the true $\Lambda$ and $\mW$. 

We apply recent work by Tropp~\cite{tropp2012user} and Gittens and Tropp~\cite{gittens2011tail} on the spectrum of sums of random matrices to answer these questions. We were motivated to use these tools by Section 7 in Gittens and Tropp~\cite{gittens2011tail}, which studies the spectrum of a finite sample estimate of a covariance matrix for a Gaussian random vector. In the present case, the gradient vector $\nabla_\vx f(\vx)$ is a deterministic function of $\vx$. However, if $\vx_j$ are drawn independently at random according to the density $\rho$, then we can interpret $\nabla_\vx f(\vx_j)$ as a random draw from an unknown density. This is a standard interpretation of Monte Carlo techniques for integration~\cite{mcbook}. In principle, our analysis approach could apply to model reduction of high-dimensional systems that use Grammian matrices~\cite{Antoulas2005}. 

\begin{theorem}
\label{thm:prob}
Assume that $\|\gf\|\leq L$ for all $\vx\in\sX$. Then for $0<\varepsilon\leq 1$,
\begin{equation}
\label{eq:upper}
\Prob{\hlambda_k \geq (1+\varepsilon)\lambda_k}
\;\leq\;
(m-k+1)\exp\left(\frac{-N\lambda_k\varepsilon^2}{4L^2}\right),
\end{equation}
and
\begin{equation}
\label{eq:lower}
\Prob{\hlambda_k \leq (1-\varepsilon)\lambda_k}
\;\leq\;
k\exp\left(\frac{-N\lambda_k^2\varepsilon^2}{4\lambda_1L^2}\right). 
\end{equation}
\end{theorem}

The key to establishing Theorem~\ref{thm:prob} is a matrix Bernstein inequality from Theorem 5.3 in Gittens and Tropp~\cite{gittens2011tail}.  When we apply this concentration result, we set
\begin{equation}
\mX_j = \gf_j^{} \gf_j^T.
\end{equation}
Thus, each $\mX_j$ is an independent random sample of a matrix from the same distribution. Under this notion of randomness,
\begin{equation} 
\Exp{ \mX_j } = \intrho{\gf_j^{} \gf_j^T} = \intrho{\gf \gf^T} = \mC.
\end{equation}
For completeness, we restate Theorem 5.3 from~\cite{gittens2011tail}.

\begin{theorem}[Eigenvalue Bernstein Inequality for Subexponential Matrices, Theorem 5.3~\cite{gittens2011tail}] \label{thm:tropp}
 Consider a finite sequence $\{ \mX_j \}$ of independent, random, self-adjoint matrices with dimension $n$, all of which satisfy the subexponential moment growth condition 
 \[ \Exp{ \mX_j^m } \preceq \frac{m!}{2} B_{}^{m-2} \mSigma_{j}^2 \qquad \text{ for } m = 2, 3, 4, \ldots \]
 where $B$ is a positive constant and $\mSigma_j^2$ are positive-semidefinite matrices.  Given an integer $k \le n$, set 
 \[ \mu_k = \lambda_k \biggl( \sum_j \Exp{\mX_j} \biggr). \]
 Choose $\mV_+$ as an orthogonal matrix of size $n \times n-k+1$ that satisfies 
 \[ \mu_k = \lambda_{\max}\biggl( \sum_j \mV_+^T (\mathbb{E} \mX_j^{}) \mV_+^{}\biggr), \]
 and define 
 \[ \sigma_k^2 = \lambda_{\max}\biggl( \sum_j \mV_+^T \mSigma_j^2 \mV_+^{}\biggr). \]
 Then, for any $\tau \ge 0$, 
 \[ \Prob{ \lambda_k \biggl( \sum_j \mX_j \biggr) \ge \mu_k + \tau } \le 
    \begin{cases} (n - k + 1) \, \exp( - \tau^2/(4\sigma_k^2 )), & \tau \le \sigma_k^2/B,  \\ 
                  (n - k + 1) \, \exp( - \tau / (4B) ), & \tau \ge \sigma_k^2 / B.
    \end{cases}
 \]
\end{theorem}

\begin{proof} \emph{(Theorem~\ref{thm:prob}.)}
We begin with the upper estimate \eqref{eq:upper}. First note that
\begin{equation}
\Prob{\lambda_k(\hmC) \geq \lambda_k(\mC) + t}
\;=\;
\Prob{\lambda_k\left(\sum_{j=1}^N \gf_j^{}\gf_j^T\right) \geq N\lambda_k + Nt}. 
\end{equation}
In this form we can apply Theorem \ref{thm:tropp}. We check that the bound on the gradient's norm implies that the matrix $\gf\gf^T$ satisfies the subexponential growth condition:
\begin{equation}
\begin{aligned}
\intrho{\left(\gf\gf^T\right)^p} 
&= \intrho{(\gf^T\gf)^{p-1}\,\gf\gf^T}\\
&\preceq \left(L^2\right)^{p-1}\,\intrho{\gf\gf^T}\\
&\preceq \frac{p!}{2}\,\left(L^2\right)^{p-2}\,(L^2\mC).
\end{aligned}
\end{equation}
Next we set 
\begin{equation}
\mu_k \;=\; \lambda_k\left(\sum_{j=1}^N \intrho{\gf_j^{} \gf_j^T } \right)
\;=\; N\lambda_k,
\end{equation}
where we simplified using the identically distributed samples of $\vx_j$. Choose $\mW_+=\mW(:,k:m)$ to be the last $m-k+1$ eigenvectors of $\mC$, and note that 
\begin{equation}
\lmax\left(\sum_{j=1}^N 
\mW_+^T\left(\intrho{\gf\gf^T}\right)\mW_+
\right) \;=\; N\lmax(\mW_+^T\mC\mW_+) \;=\; N\lambda_k \;=\;\mu_k,
\end{equation}
as required by Theorem~\ref{thm:tropp}. Define
\begin{equation}
\sigma^2_k \;=\;
\lmax\left(\sum_{j=1}^N \mW_+^T(L^2\mC)\mW_+\right) 
\;=\; NL^2\, \lmax\left(\mW_+^T\mC\mW_+\right)
\;=\; NL^2\lambda_k.
\end{equation}
With these quantities, Theorem~\ref{thm:tropp} states
\begin{equation}
\Prob{\lambda_k\left(\sum_{j=1}^N \gf_j^{} \gf_j^T \right) \geq N\lambda_k + Nt}
\;\leq\;
(m-k+1)\exp\left(\frac{-(Nt)^2}{4 \sigma_k^2}\right)
\end{equation}
when $Nt \leq \sigma_k^2/L^2$. Applying this theorem with $t=\varepsilon\lambda_k$, $\varepsilon\leq 1$, and the computed $\sigma_k^2=NL^2\lambda_k$ yields the upper estimate \eqref{eq:upper}.

For the lower estimate,
\begin{equation}
\begin{aligned}
& \Prob{\lambda_k(\hmC) \leq \lambda_k(\mC) - t} \\
& \qquad =
\Prob{-\lambda_k(\hmC) \geq -\lambda_k(\mC) + t} \\
& \qquad =
\Prob{-\lambda_k\left(\sum_{j=1}^N \gf_j^{}\gf_j^T\right) \geq -N\lambda_k(\mC) + Nt} \\
& \qquad =
\Prob{\lambda_{m-k+1}\left(\sum_{j=1}^N \left(-\gf_j^{}\gf_j^T\right)\right) \geq N\lambda_{m-k+1}(-\mC) + Nt} \\
& \qquad =
\Prob{\lambda_{k'}\left(\sum_{j=1}^N \left(-\gf_j^{}\gf_j^T\right)\right) \geq N\lambda_{k'}(-\mC) + Nt},
\end{aligned}
\end{equation}
for $k'=m-k+1$. We can now apply Theorem~\ref{thm:tropp} again. The subexponential growth condition is satisfied since
\begin{equation}
\intrho{\left(-\gf\gf^T\right)^p} \;\preceq\;
\intrho{\left(\gf\gf^T\right)^p} \;\preceq\;
\frac{p!}{2}\,\left(L^2\right)^{p-2}\,(L^2\mC).
\end{equation}
Set
\begin{equation}
\mu_{k'} \;=\; 
\lambda_{k'}\left(\sum_{j=1}^N \intrho{ (-\gf_j^{} \gf_j^{T}) } \right)
\;=\; N\lambda_{k'}(-\mC).
\end{equation}
Set $\mW_+=\mW(:,1:k)$ to be the first $k$ eigenvectors of $\mC$, and note that
\begin{equation}
\begin{aligned}
\lmax\left(\sum_{j=1}^N\mW_+^T\left(
\intrho{\left(-\gf_j^{}\gf_j^T\right)}
\right)\mW_+\right) 
&=
N\lmax\left(-\mW_+^T\mC\mW_+\right) \\
&= 
N(-\lambda_k(\mC)) \\
&= N\lambda_{m-k+1}(-\mC) \\
&= N\lambda_{k'}(-\mC),
\end{aligned}
\end{equation}
as required by Theorem \ref{thm:tropp}. Set
\begin{equation}
\sigma_{k'}^2 \;=\; 
\lmax\left(\sum_{j=1}^N \mW_+^T(L^2\mC)\mW_+\right)
\;=\; NL^2 \lmax\left(\mW_+^T\mC\mW_+\right)
\;=\; NL^2\lambda_1.
\end{equation}
Theorem \ref{thm:tropp} states
\begin{equation}
\Prob{\lambda_{k'}\left(\sum_{j=1}^N \left(-\gf_j^{}\gf_j^T\right)\right) \geq N\lambda_{k'}(-\mC) + Nt}
\;\le\; k\exp\left(\frac{-(Nt)^2}{4 \sigma_{k'}^2}\right)
\end{equation}
when $Nt \leq \sigma_{k'}^2/L^2$. Plug in the computed quantities with $t=-\varepsilon\lambda_{k'}(-\mC)=\varepsilon\lambda_k(\mC)$ to achieve the lower estimate \eqref{eq:lower}. Note that the condition $\varepsilon\leq 1\leq\lambda_1/\lambda_{k'}$ allows us to apply Theorem \ref{thm:tropp}. 
\end{proof}

Next we use this result to derive a lower bound on the number of gradient samples needed for relative accuracy of $\varepsilon$. Recall the definition of \emph{big omega} notation that $a=\Omega(b)$ means $a\geq cb$ for some positive constant $c$.

\begin{corollary}
\label{thm:N}
Let $\kappa_k=\lambda_1/\lambda_k$. Then for $\varepsilon\in(0,1]$, 
\begin{equation} 
\label{eq:eigN}
N \;=\; \Omega\left(\frac{L^2\kappa_k^2}{\lambda_1\varepsilon^2}\log(m)\right)
\end{equation}
implies $|\hlambda_k - \lambda_k| \leq \varepsilon\lambda_k$ with high probability.
\end{corollary}

\begin{proof}
Starting with the upper estimate from Theorem \ref{thm:prob}, if
\begin{equation}
N \;\geq\;
\frac{4L^2}{\lambda_k\varepsilon^2}(\beta+1)\log(m) 
\;\geq\;
\frac{4L^2}{\lambda_k\varepsilon^2}(\beta\log(m)+\log(m-k+1)),
\end{equation}
then 
\begin{equation}
\Prob{\hlambda_k \geq (1+\varepsilon)\lambda_k} \leq m^{-\beta}.
\end{equation}
Similarly for the lower estimate from Theorem \ref{thm:prob}, if 
\begin{equation}
N \;\geq\;
\frac{4L^2\lambda_1}{\lambda_k^2\varepsilon^2}(\beta+1)\log(m)
\;\geq\;
\frac{4L^2\lambda_1}{\lambda_k^2\varepsilon^2}(\beta\log(m)+\log(k)),
\end{equation}
then 
\begin{equation}
\Prob{\hlambda_k \leq (1-\varepsilon)\lambda_k} \leq m^{-\beta}.
\end{equation}
Setting $\kappa_k=\lambda_1/\lambda_k$ and taking
\begin{equation}
N \;\geq\; (\beta+1)\frac{4L^2\kappa_k^2}{\lambda_1\varepsilon^2}\log(m)
\end{equation}
satisfies both conditions. 
\end{proof}

We can combine results from Golub and Van Loan~\cite[Chapter 8]{golub1996matrix} with results from Tropp~\cite{tropp2012user} to obtain an estimate of the distance between the subspace defined by the eigenvectors $\mW_1$ and the subspace defined by the eigenvectors $\hmW_1$. This requires a different matrix Bernstein inequality in the form of Theorem 6.1 from Tropp~\cite{tropp2012user}, which we restate below. When we apply the theorem, $\mX_j = \gf_j^{} \gf_j^T - \mC$, that is, the random matrix samples are the deviation of the $j$th sampled gradient outer product from the matrix $\mC$. 

\begin{theorem}[Matrix Bernstein: bounded case, Theorem 6.1~\cite{tropp2012user}]
\label{thm:tropp-zero}
Consider a finite sequence $\{ \mX_j \}$ of independent, random, self-adjoint matrices with dimension $n$. Assume that 
\[ \Exp{ \mX_j } = 0 \quad \text{ and } \quad \lmax( \mX_j ) \le R \quad \text{almost surely}. \]
Compute the norm of the total variance,
 \[ \sigma^2 := \left \| \sum_{j} \Exp{ \mX_j^2 } \right\|. \]
Then the following inequality holds for all $\tau \ge 0$: 
\[ \Prob{ \lmax\biggl( \sum_{j} \mX_j \biggr) \ge \tau } 
	\le \begin{cases} 
			n \exp(-3\tau^2/(8\sigma^2)), & \tau \le \sigma^2 / R, \\
			n \exp(-3\tau/(8R)), & \tau > \sigma^2 / R.
		\end{cases} 
\]
\end{theorem}

\begin{theorem}
\label{thm:maxeig}
Let $\varepsilon>0$. Assume $\|\gf\|\leq L$ for all $\vx\in\sX$. Define the variance
\begin{equation}
\label{eq:nu}
\nu^2 \;=\; \left\| \int \left(\nabla_\vx f\,\nabla_\vx f^T -\mC\right)^2\,\rho\,dx \right\|,
\end{equation}
and assume $\nu^2>0$. Then
\begin{equation}
\label{eq:2normC}
\Prob{\|\hmC-\mC\| \geq \varepsilon\|\mC\|} 
\;\leq\; \left\{
\begin{array}{ll}
2 m\,\exp\left( \frac{-3N\lambda_1^2\varepsilon^2}{8\nu^2} \right),
& \mbox{ if $\varepsilon\leq \nu^2/(\lambda_1L^2)$,} \\
2 m\,\exp\left( \frac{-3N\lambda_1\varepsilon}{8L^2} \right),
& \mbox{ if $\varepsilon > \nu^2/(\lambda_1L^2)$.}
\end{array}
\right.
\end{equation}
\end{theorem}

\begin{proof}
Observe that 
\begin{equation}
\label{eq:probC}
\begin{aligned}
\Prob{\|\hmC-\mC\| \geq t } 
  & = \Prob{ \lmax ( \hmC - \mC ) \geq t  \text{ or }  \lmax ( \mC - \hmC ) \geq t } \\
  & \le \Prob{ \lmax ( \hmC - \mC ) \geq t } + \Prob{ \lmax ( \mC - \hmC ) \geq t } \\
  & = \Prob{ \lmax\left( \sum_{j=1}^N \left(\gf_j^{}\gf_j^T-\mC \right) \right) \geq Nt} \\
  & \qquad + \Prob{ \lmax\left( \sum_{j=1}^N \left(\mC - \gf_j^{}\gf_j^T \right) \right) \geq Nt}.
\end{aligned}  
\end{equation} 
Note that both
\begin{equation}
\intrho{\left(\gf\gf^T - \mC\right)} \;=\; \intrho{\left(\mC - \gf\gf^T \right)}
\;=\; \mathbf{0}.
\end{equation}
Since $\mC$ is positive semidefinite and $\|{\gf}\| \le L$,
\begin{equation}
\begin{aligned}
\lmax(\gf \gf^T - \mC) 
&= \underset{\|\vv\|=1}{\operatorname{max}}\; \vv^T\left(\gf\gf^T - \mC\right)\vv \\
& \leq \underset{\|\vv\|=1}{\operatorname{max}}\; \vv^T\left(\gf\gf^T\right)\vv \;\leq\;
L^2.
\end{aligned}
\end{equation}
This bound also holds for $\lmax(\mC - \gf \gf^T)$, since
\begin{equation}
\begin{aligned}
\lmax(\mC - \gf \gf^T) 
&= \underset{\|\vv\|=1}{\operatorname{max}}\; \vv^T\left(\mC - \gf\gf^T\right)\vv \\
& \leq \underset{\|\vv\|=1}{\operatorname{max}}\; \vv^T\,\mC\,\vv \;\leq\; \|\mC\| \;\leq\;
L^2.
\end{aligned}
\end{equation}
Thus, the upper-bound $R$ in Theorem \ref{thm:tropp-zero} is $L^2$. The variance parameter $\sigma^2$ is
\begin{equation}
\sigma^2 
 = \left\| \left( \sum_{j=1}^N \intrho{ (\gf_j \gf_j^T - \mC)^2 } \right) \right\| 
 = N \left\| \intrho{ (\gf \gf^T - \mC)^2 } \right\| 
 = N\,\nu^2.
\end{equation}             
Assume $\varepsilon\leq \nu^2/(\lambda_1L^2)$. Then $N\lambda_1\varepsilon\leq N\nu^2/L^2$, and we can apply the upper branch of Theorem \ref{thm:tropp-zero} to the two terms at the end of \eqref{eq:probC} with $\tau=\lambda_1\varepsilon=\|\mC\|\varepsilon$, which produces the bound's upper branch in \eqref{eq:2normC}. 
Now assume $\varepsilon >\nu^2/(\lambda_1L^2)$. Similarly, $N\lambda_1\varepsilon\geq N\nu^2/L^2$, and we can apply the lower branch of Theorem \ref{thm:tropp-zero} with the same $t=\lambda_1\varepsilon$ to \eqref{eq:probC} to produce the lower branch of \eqref{eq:2normC}.
\end{proof}




This result leads to a lower bound on the number of samples needed for a small relative error in $\hmC$ in the matrix 2-norm; compare the following corollary to Corollary \ref{thm:N}. 

\begin{corollary}
\label{cor:maxeigN}
Let $\varepsilon>0$, and define
\begin{equation}
\label{eq:delta}
\delta \;=\; \max\,\left(\frac{\nu^2}{\lambda_1\varepsilon},\, L^2\right).
\end{equation}
Then
\begin{equation}
\label{eq:maxeigN}
N \;=\; \Omega\left(\frac{\delta}{\lambda_1\varepsilon}\,\log(2m)\right)
\end{equation}
implies that $\|\hmC-\mC\|\leq\varepsilon\|\mC\|$ with high probability.
\end{corollary}

\begin{proof}
We consider the two cases of $\varepsilon$ from Theorem \ref{thm:maxeig}. Assume $\varepsilon\leq \nu^2/(\lambda_1L^2)$, so that $N\lambda_1\varepsilon\leq N\nu^2/L^2$, we follow the reasoning in the proof of Corollary \ref{thm:N} with the upper branch of the bound in Theorem \ref{thm:maxeig} to get
\begin{equation}
\label{eq:lb1}
N\;\geq\; \frac{8}{3} (\beta+1)\frac{\nu^2}{\lambda_1^2\varepsilon^2}\log(2m).
\end{equation}
Similarly, if $\varepsilon >\nu^2/(\lambda_1L^2)$, then the lower branch from the bound in Theorem \ref{thm:maxeig} produces
\begin{equation}
\label{eq:lb2}
N\;\geq\; \frac{8}{3} (\beta+1)\frac{L^2}{\lambda_1\varepsilon}\log(2m).
\end{equation} 
Using $\delta$ from \eqref{eq:delta} chooses the larger lower bound between \eqref{eq:lb1} and \eqref{eq:lb2}.
\end{proof}

We can combine Corollary \ref{cor:maxeigN} with~\cite[Corollary 8.1.11]{golub1996matrix} to control the error in the estimated subspace defined by $\hmW_1$. We quantify this error by the distance between the subspace defined by $\mW_1$ and the subspace defined by $\hmW_1$. Recall the definition of the distance between subspaces~\cite{stewart1973},
\begin{equation}
\dist{\mW_1}{\hmW_1} \;=\;
\|\mW_1\mW_1^T - \hmW_1\hmW_1^T\| \;=\;
\|\mW_1^T\hmW_2\|.
\end{equation}

\begin{corollary}
\label{cor:evecs}
Let $\varepsilon>0$ be such that
\begin{equation}
\label{eq:epsbnd}
\varepsilon \;\leq\; \min\,
\left(1,\,(\lambda_{n}-\lambda_{n+1})/(5\lambda_1)\right),
\end{equation}
and choose $N$ according to Corollary \ref{cor:maxeigN}. Then
\begin{equation}
\dist{\mW_1}{\hmW_1} \;\leq\; \frac{4\lambda_1\varepsilon}{\lambda_n-\lambda_{n+1}},
\end{equation}
with high probability.
\end{corollary}

\begin{proof}
Let $\mE = \hmC-\mC$. For $\varepsilon$ in \eqref{eq:epsbnd} with $N$ chosen according to Corollary \ref{cor:maxeigN}, we have
\begin{equation}
\|\mE\|\;\leq\; \varepsilon\|\mC\| \;=\; \varepsilon\lambda_1
\;\leq\; (\lambda_n-\lambda_{n+1})/5, 
\end{equation}
with high probability. Under this condition on $\|\mE\|$, \cite[Corollary 8.1.11]{golub1996matrix} states
\begin{equation}
\dist{\mW_1}{\hmW_1} 
\;\leq\; 
\frac{4\|\mE\|}{\lambda_n-\lambda_{n+1}}
\;\leq\; 
\frac{4\lambda_1\varepsilon}{\lambda_n-\lambda_{n+1}},
\end{equation}
as required.
\end{proof}

Corollary \ref{cor:evecs} shows that control of the eigenvalues implies control of the subspace generated by the eigenvectors. However, the error in the estimated subspace is inversely proportional to the gap between the smallest eigenvalue associated with the active subspace and the largest eigenvalue associated with the inactive subspace. This implies, for example, if the gap between the second and third eigenvalues is larger than the gap between the first and second, then estimates of a two-dimensional active subspace are more accurate than estimates of a one-dimensional active subspace. 


\subsection{Approximate gradients}
\label{sec:approxgrad}
Many modern simulations have subroutines for estimating gradients with, e.g., adjoint methods~\cite{Bryson75,Borzi2012} or algorithmic differentiation~\cite{Griewank00}. However, legacy codes or simulations that couple multiple codes might not have such gradient capabilities. When there is no subroutine for gradients, finite difference approximations may suffice when $m$ is not too large and $f$ is neither too expensive nor too noisy. Recent work characterizes the gradient when function evaluations contain noise~\cite{More2011,More2014}. 

Next, we extend the bounds on errors in the estimated eigenpairs to the case when the gradients are computed with some error. The gradient error model we analyze depends on a parameter that controls the amount of error. Let $\vg(\vx)$ denote the approximate gradient computed at $\vx\in\sX$. We assume that 
\begin{equation}
\label{eq:fdbnd}
\|\vg(\vx) - \gf(\vx)\| \;\leq\; \sqrt{m}\gamma_h,\quad \vx\in\sX,
\end{equation}
where
\begin{equation}
\lim_{h\rightarrow 0} \gamma_h \;=\; 0.
\end{equation}
The parameter $h$ may be a finite difference parameter, the grid spacing in a continuous adjoint computation, or the solver tolerance for a discrete adjoint computation.

Define the symmetric positive semidefinite matrix $\mG$ and its eigenvalue decomposition
\begin{equation}
\label{eq:g}
\mG \;=\; \intrho{\vg\,\vg^T} \;=\;\mU\Theta\mU^T,\quad
\Theta=\text{diag}\,(\theta_1,\dots,\theta_m),
\end{equation}
and define its random sample approximation
\begin{equation}
\label{eq:ghat}
\hmG \;=\; \frac{1}{N}\sum_{j=1}^N\vg_j\,\vg_j^T \;=\;\hmU\hTheta\hmU^T,\quad
\hTheta=\text{diag}\,(\htheta_1,\dots,\htheta_m),
\end{equation}
where $\vg_j=\vg(\vx_j)$ for $\vx_j$ drawn independently from $\rho$. With these quantities defined, we have the following lemma.

\begin{lemma}
\label{lem:2norm}
Let $\|\nabla_\vx f\|\leq L$ for all $\vx\in\sX$. The norm of the difference between $\hmC$ and $\hmG$ is bounded by
\begin{equation}
\|\hmC - \hmG\| \;\leq\; (\sqrt{m}\gamma_h+2L)\sqrt{m}\gamma_h.
\end{equation}
\end{lemma}

\begin{proof}
Let $\vg=\vg(\vx)$ and $\gf=\gf(\vx)$. First observe
\begin{equation}
\|\vg+\gf\| = 
\|\vg-\gf + 2\gf\| \leq
\|\vg-\gf\| + 2\|\gf\| \leq
\sqrt{m}\gamma_h + 2L.
\end{equation}
Next,
\begin{equation}
\begin{aligned}
\|\vg\,\vg^T-\gf\gf^T\|
&= \frac{1}{2} \|(\vg+\gf)(\vg-\gf)^T+(\vg-\gf)(\vg+\gf)^T\|\\
&\leq \|(\vg+\gf)(\vg-\gf)^T\|\\
&\leq (\sqrt{m}\gamma_h +2L)\sqrt{m}\gamma_h.
\end{aligned}
\end{equation}
Then,
\begin{equation}
\begin{aligned}
\|\hmG-\hmC\|
&= \left\|\frac{1}{N}\sum_{j=1}^N \vg_j\vg_j^T - \frac{1}{N}\sum_{j=1}^N \gf_j^{}\gf_j^T\right\|\\
&\leq \frac{1}{N} \sum_{j=1}^N \|\vg_j\vg_j^T-\gf_j^{}\gf_j^T\|\\
&\leq \sqrt{m}\gamma_h(\sqrt{m}\gamma_h + 2L). 
\end{aligned}
\end{equation}
\end{proof}

We combine Lemma \ref{lem:2norm} with Corollary \ref{thm:N} to study the error in the eigenvalues of the random sample estimate with approximate gradients. 

\begin{theorem}
For $\varepsilon\in(0,1]$, if $N$ is chosen as \eqref{eq:eigN}, then the difference between $\lambda_k$ in \eqref{eq:eigdecomp} and the eigenvalue $\htheta_k$ from \eqref{eq:ghat} is bounded as
\begin{equation}
|\lambda_k - \htheta_k| \;\leq\; \varepsilon\lambda_k +
\sqrt{m}\gamma_h(\sqrt{m}\gamma_h + 2L),
\end{equation}
with high probability. 
\end{theorem}

\begin{proof}
Observe that
\begin{equation}
|\lambda_k - \htheta_k| \leq |\lambda_k-\hlambda_k| + |\hlambda_k - \htheta_k|.
\end{equation}
Apply Corollary \ref{thm:N} to the first term. The second term follows from~\cite[Corollary 8.1.6]{golub1996matrix} combined with Lemma \ref{lem:2norm}, since
\begin{equation}
|\htheta_k-\hlambda_k| =
|\lambda_k(\hmG) - \lambda_k(\hmC)| \leq
\|\hmG-\hmC\| \leq
\sqrt{m}\gamma_h(\sqrt{m}\gamma_h + 2L).
\end{equation}
\end{proof}

The bias in the finite sample eigenvalue estimates using approximate gradients goes to zero at the same rate as the error in the approximate gradient. Next we attend to the error in the active subspace computed with Monte Carlo and approximate gradients. 

\begin{theorem}
\label{thm:subspaceer}
Choose $\varepsilon>0$ such that
\begin{equation}
\varepsilon \;<\; \min\left(
1, \; \frac{\lambda_n-\lambda_{n+1}}{5\lambda_1}, \; 
\frac{\lambda_n - \lambda_{n+1}}{\lambda_{n}+\lambda_{n+1}}
\right)
\end{equation}
Choose $N$ so that it satisfies both \eqref{eq:maxeigN} and \eqref{eq:eigN} with $k=n+1$. Choose $h$ small enough so that
\begin{equation}
\label{eq:hcond}
\sqrt{m}\gamma_h(\sqrt{m}\gamma_h + 2L) \leq \frac{(1-\varepsilon)\lambda_n-(1+\varepsilon)\lambda_{k+1}}{5}.
\end{equation}
Then 
\begin{equation}
\dist{\hmU_1}{\mW_1} \;\leq\;
\frac{4\sqrt{m}\gamma_h(\sqrt{m}\gamma_h+2L)}{(1-\varepsilon)\lambda_n - (1+\varepsilon)\lambda_{n+1}} \;+\;
\frac{4\lambda_1}{\lambda_n - \lambda_{n+1}},
\end{equation}
with high probability. 
\end{theorem}

\begin{proof}
The conditions on $N$ and $\varepsilon$ imply $|\hlambda_{n+1}-\lambda_{n+1}|\leq\varepsilon\lambda_{n+1}$ with high probability due to Corollary \ref{thm:N}. Examining \eqref{eq:eigN}, we see that if $N$ is large enough to estimate $\lambda_{n+1}$, then $N$ is large enough to estimate $\lambda_n$, so $|\hlambda_n-\lambda_n|\leq\varepsilon\lambda_n$ with high probability, too. Then
\begin{equation}
\begin{aligned}
\lambda_n-\lambda_{n+1} &= |\lambda_n-\lambda_{n+1}|\\
&\leq |\lambda_n-\hlambda_n| + |\hlambda_{n+1}-\lambda_{n+1}| + (\hlambda_n-\hlambda_{n+1})\\
&\leq \varepsilon\lambda_n + \varepsilon\lambda_{n+1} + (\hlambda_n-\hlambda_{n+1}),
\end{aligned}
\end{equation}
with high probability. Rearranging this inequality yields
\begin{equation}
\label{eq:thegap}
\hlambda_n-\hlambda_{n+1} 
\;\geq\;
(1-\varepsilon)\lambda_n - (1+\varepsilon)\lambda_{n+1}.
\end{equation}
This relates the gap between the eigenvalue estimates to the gap between the true eigenvalues. The condition on $\varepsilon$ ensures that
\begin{equation}
(1-\varepsilon)\lambda_n - (1+\varepsilon)\lambda_{n+1}
\;>\; 0.
\end{equation}
Next,
\begin{equation}
\dist{\hmU_1}{\mW_1} \;\leq\; \dist{\hmU_1}{\hmW_1} + \dist{\hmW_1}{\mW_1}.
\end{equation}
The second term on the right is bounded in Corollary \ref{cor:evecs} under the assumptions on $N$ and $\varepsilon$. The condition \eqref{eq:hcond} on $h$ and \eqref{eq:thegap} imply
\begin{equation}
\sqrt{m}\gamma_h(\sqrt{m}\gamma_h + 2L) \leq \frac{\hlambda_n-\hlambda_{k+1}}{5}.
\end{equation}
Then \cite[Corollary 8.1.11]{golub1996matrix} implies 
\begin{equation}
\dist{\hmU_1}{\hmW_1} \;\leq\;
\frac{4}{\hlambda_n - \hlambda_{n+1}}\|\hmG-\hmC\|. 
\end{equation}
Combining this with \eqref{eq:thegap} and the bound from Lemma \ref{lem:2norm} yields the result. 
\end{proof}

In summary, the eigenvalues and the active subspace approximated with Monte Carlo and approximate gradients are well-behaved. The error bounds include a term that goes to zero like the error in the approximate gradient and a term that behaves like the finite sample approximation with exact gradients. Note that the error bound on the subspace estimate depends on both the gap between $\lambda_n$ and $\lambda_{n+1}$ and a smaller gap that depends on $\varepsilon$.

\section{Practical approach to computation}
\label{sec:prac}

The bounds we present in Section \ref{sec:comp} provide a theoretical foundation for understanding the behavior of the Monte Carlo estimates. However, many of the quantities in the bounds may not be known a priori---such as the maximum norm of the gradient $L$ and the true eigenvalues of the matrix $\mC$. In this section we offer a practical recipe guided by the insights from the theory. We caution that the following approach, which relies on a nonparametric bootstrap, can perform poorly for badly-behaved functions. For example, one could be unlucky and sample the gradient in regions that are not representative of the gradient over the entire domain; the bootstrap uses only the $N$ samples used to compute the eigenpair estimates. Also, errors in the gradients could produce poor approximations of the eigenvalues and subspaces; we show an example of this in Section \ref{sec:quad}. Nevertheless, we have used the following approach on several problems in practice to reveal low-dimensional structure in complex functions of several variables coming from engineering simulations~\cite{constantine2014active,constantine2014discovering,Lukaczyk2014}.

The first objective is to estimate the eigenvalues and a measure of the estimates' variability from the finite samples. Suppose one wishes to estimate the first $k$ eigenpairs from the matrix $\mC$. Practical considerations guide the choice of $k$. For example, if one wishes to build a response surface approximation of $f$ on a low-dimensional domain, then five or six dimensions might be the most one can afford given the cost of computing $f(\vx)$. Hence $k$ might be seven or eight to allow the possibility of finding a gap that indicates a sufficiently low-dimensional approximation. If a gap is not present in the first $k$ eigenvalues, then $f$ may not be amenable to dimension reduction via active subspaces for the desired purpose. 

We recommend choosing the number $N$ of independent gradient samples as
\begin{equation}
\label{eq:pracN}
N \;=\; \alpha\,k\,\log(m),
\end{equation}
where $\alpha$ is a multiplier between 2 and 10. Taking at least $k$ samples means that $\mC$ is a sum of $k$ rank-one matrices and thus has a rank of at most $k$. This allows the possibility of estimating $k$ non-zero eigenvalues. The $\log(m)$ term follows from the bounds in Theorem \ref{thm:prob}. The $\alpha$ between 2 and 10 is an ad hoc multiplicative factor that we have used on several problems. In principle, $\alpha k$ is meant to model the contribution from the unknown terms $L$, $\kappa_k$, $\nu^2$, and $\lambda_1$ in \eqref{eq:eigN} and \eqref{eq:maxeigN}. It is likely that the combination of these terms with the $\varepsilon^{-1}$ is greater than $10k$ for small $\varepsilon$. However, the Bernstein inequalities used to derive the lower bounds on $N$ in Corollaries \ref{thm:N} and \ref{cor:maxeigN} are also conservative. One can also assess if $N$ is large enough a posteriori by examining the bootstrap intervals described below. 

%

We form $\hmC$ using the samples of the gradient as in \eqref{eq:compC}, and then compute its eigenvalue decomposition. We expect that computing the full eigendecomposition is much cheaper than computing the gradient samples. A function of a thousand variables produces $\hmC$ with dimension thousand-by-thousand. Full eigendecompositions for matrices this size are computed in seconds on modern laptops.  


We suggest computing bootstrap intervals for the eigenvalues, which involves computing the eigendecompositions of several matrices the size of $\hmC$. The bootstrap creates replicates by (i) sampling with replacement from the set of gradient samples, (ii) computing the replicate $\hmC^\ast$, and (iii) computing its eigenvalue decomposition. The collection of eigenvalue replicates is used to estimate bounds on the true eigenvalues. Efron and Tibshirani use the bootstrap to get empirical density functions of estimated eigenvalues from a covariance matrix in section 7.2 of their book~\cite{efron1994introduction}. Chapter 3 of Jolliffe's book~\cite{jolliffe2005principal} also comments on the bootstrap approach for estimating eigenvalues and eigenvectors of a covariance matrix from independent samples. The bootstrap estimates of the standard error and confidence intervals for the eigenvalues may be biased, but this bias decreases as the number $N$ of samples increases. Since these estimates may be biased, we refer to them as \emph{bootstrap intervals} instead of \emph{confidence intervals}. 

Corollary \ref{cor:evecs} says that the error in the estimated subspace depends inversely on the gap between the eigenvalues scaled by the largest eigenvalue. The key to accurately approximating the subspace is to look for gaps in the eigenvalues; this is consistent with standard perturbation theory for eigenvector computations~\cite{stewart1973}. For example, if there is a larger gap between the third and fourth eigenvalues than between the second and third, then estimates of the three-dimensional subspace are more accurate than estimates of the two-dimensional subspace. This contrasts with heuristics for choosing the dimension of the subspace in (i) model reduction based on the proper orthogonal decomposition~\cite{berkooz1993pod} and (ii) dimension reduction based on principal component analysis~\cite{jolliffe2005principal}. In these cases, one chooses the dimension of the subspace by a threshold on the magnitude of the eigenvalues---e.g., so that the sum of retained eigenvalues exceeds some proportion of the sum of all eigenvalues. 
To accurately approximate the active subspace, the most important quantity is the spectral gap, which indicates a separation between scales. To tease out the spectral gap, plot the estimated eigenvalues and their respective upper and lower bootstrap intervals; a gap between subsequent intervals offers confidence of a spectral gap and, hence, the presence of an active subspace. In Section \ref{sec:experiments}, we show several examples of such plots (Figures \ref{fig:quad0}, \ref{fig:quad1}, \ref{fig:pde0}, \ref{fig:pde1}). 

One should also consider the intent of the dimension reduction when choosing the dimension of the active subspace. For example, if the goal is to approximate a function of $m$ variables by a surrogate model of $n$ variables---as in~\cite{constantine2014active}---then one may be limited to $n$'s small enough to permit surrogate construction. Suppose the largest $n$ one is willing to use is $n=n_{\text{max}}=5$, but there is no gap between consecutive eigenvalues from $\lambda_1$ to $\lambda_6$. Then subspace-based dimension reduction may not be an appropriate tool, and one should consider searching for other types of exploitable structure in the model.  

Assuming we have chosen $n$, we wish to study the variability in the active subspace due to finite sampling; we again turn to the bootstrap. In particular, for each replicate $\hmW^\ast$ of the eigenvectors, we compute $\operatorname{dist}(\text{ran}(\hmW_1),\text{ran}(\hmW_1^\ast))$. One can examine the bootstrap intervals of this quantity to study the stability of the subspace. Recall that the distance between subspaces is bounded above by 1, so a bootstrap interval whose values are close to 1 indicates a poorly approximated active subspace. Figures \ref{fig:quad0}, \ref{fig:quad2}, \ref{fig:pde0}, and \ref{fig:pde1} show examples of plotting this metric for the stability of the subspace; the first two Figures also compare the measure of stability to the true error in the active subspace. 

\subsection{A step-by-step procedure}
We summarize the practical approach to approximating the active subspace with bootstrap intervals. What follows is a modification of the procedure outlined at the beginning of Section \ref{sec:comp} including our suggestions for parameter values. This procedure assumes the user has decided on the number $k$ of eigenvalues to examine. 
\begin{enumerate}
\item Choose $N = \alpha k \log(m)$, where $\alpha$ is a multiplier between 2 and 10, and choose $N_{\text{boot}}$ between 100 and 10000.
\item Draw $N$ samples $\{\vx_j\}$ independently from $\rho$. For each $\vx_j$, compute $\gf_j = \gf(\vx_j)$. 
\item Compute 
\begin{equation}
\hmC \;=\;
\frac{1}{N} \sum_{j=1}^N (\nabla_\vx f_j)(\nabla_\vx f_j)^T \;=\;
\hmW\hLambda\hmW^T.
\end{equation}
\item \textbf{Bootstrap:} For $i$ from 1 to $N_{\text{boot}}$, let $\ell^i_1,\dots,\ell^i_N$ be $N$ integers drawn randomly from $\{1,\dots,N\}$ with replacement, and compute
\begin{equation}
\label{eq:bstrap}
\hmC_i^\ast \;=\;
\frac{1}{N} \sum_{j=1}^N (\nabla_\vx f_{\ell^i_j})(\nabla_\vx f_{\ell^i_j})^T \;=\;
\hmW_i^\ast\,\hLambda_i^\ast\,\left(\hmW_i^\ast\right)^T.
\end{equation}
The asterisk denotes a bootstrap replicate. Then compute the subspace distance
\begin{equation}
d_i^\ast \;=\; \dist{\hmW}{\hmW_i^\ast}.
\end{equation}
\item Compute the intervals 
\begin{equation}
\left[ 
\underset{i}{\text{min}}\;\hlambda_{j,i}^\ast,\,
\underset{i}{\text{max}}\;\hlambda_{j,i}^\ast
\right],\qquad j=1,\dots,k
\end{equation}
where $\hlambda_{j,i}^\ast$ is the $j$th diagonal from $\hLambda_i^\ast$ in \eqref{eq:bstrap}. Also compute the mean, minimum, and maximum from the set $\{d_i^\ast\}$ to estimate the subspace error. 
\item Plot the eigenvalue bootstrap intervals and look for large gaps. Choose the dimension $n$ of the active subspace corresponding to the largest eigenvalue gap. If there is no perceivable gap, then an active subspace may not be present in the first $k-1$ dimensions. 
\end{enumerate}
A few comments are in order. First, we assume the dimension $m$ of $\hmC$ is small enough so that the eigendecompositions of $\hmC$ and its bootstrap replicates are much cheaper than the samples of the gradient. Such is the case when $m$ is in the thousands (i.e., $f$ depends on thousands of input variables), and $f$ and $\gf$ come from an expensive engineering simulation. Second, we choose the bootstrap to examine the variability because we assume that sampling more gradients is not feasible. If this is not the case, i.e., if one can cheaply evaluate many more gradient samples, then one can compute Monte Carlo estimates and central limit theorem confidence intervals of the eigenvalues in place of the bootstrap estimates. Lastly, we note that the elements of $\mC$ are multivariate integrals. If $m$ is small enough (2 or 3) and evaluating $\gf$ is cheap enough, then more accurate numerical quadrature rules may perform better than the random sampling, i.e., greater accuracy for fewer samples. However, practical error estimates are more difficult to compute, since the error is due to bias instead of variance.

\section{Experiments}
\label{sec:experiments}

We apply the procedures described in Section \ref{sec:prac} to two models: (i) a quadratic function and (ii) a linear functional of the solution of a parameterized PDE. The quadratic model is simple enough to analytically derive the eigenpairs of the active subspace for thorough evaluation of the method. We study the same PDE model in~\cite[Section 5]{constantine2014active}. Gradients are available through adjoints, but the true active subspaces are not available. We support efforts for reproducible research~\cite{leveque2009python,donoho2009reproducible}; codes for the experiments in this section can be found at 
\begin{center}
\url{https://www.cs.purdue.edu/homes/dgleich/codes/compute-asm/compute-asm-code.tar.gz} 
\end{center} 
The PDE example uses the Random Field Simulation code (\url{http://www.mathworks.com/matlabcentral/fileexchange/27613-random-field-simulation}) as well as the MATLAB PDE Toolbox.

\subsection{A quadratic model}
\label{sec:quad}
Consider a quadratic function of $m=10$ variables, 
\begin{equation}
f(\vx) \;=\; \frac{1}{2}\vx^T\mA\vx, \quad \vx\in[-1,1]^{10},
\end{equation}
where $\mA$ is symmetric and positive definite. We take $\rho=2^{-10}$ on the hypercube $[-1,1]^{10}$ and zero elsewhere. The gradient is $\gf(\vx) = \mA\vx$, so
\begin{equation}
\mC \;=\; \mA\left(\intrho{\vx\,\vx^T}\right)\mA^T \;=\;
\frac{1}{3} \mA^2. 
\end{equation}
The eigenvalues of $\mC$ are the squared eigenvalues of $\mA$ divided by 3, and the eigenvectors of $\mC$ are the eigenvectors of $\mA$. 

We study three different $\mA$'s constructed from three choices for the eigenvalues: (1) exponential decay with a constant rate, (2) like the first but with a larger gap between the first and second eigenvalue, and (3) like the first with a larger gap between third and fourth eigenvalue. The three cases of eigenvalues for $\mA$ are shown in the top row of Figure \ref{fig:quad0}. Each $\mA$ has the same eigenvectors, which we generate as an orthogonal basis from a random $10\times 10$ matrix. 


To estimate the eigenvalues, we choose $N$ as in \eqref{eq:pracN} with the multiplier $\alpha=2$ and $k=6$ eigenvalues of interest, which yields $N=28$ evaluations of the gradient. The middle row of Figure \ref{fig:quad0} shows the bootstrap intervals for the first six eigenvalues along with the true eigenvalues of $\mC$. The small bootstrap intervals suggest confidence in the estimates. The gaps are apparent in the last two cases. The bottom row of Figure \ref{fig:quad0} shows bootstrap intervals on the distance between the true $k$-dimensional active subspace and the subspace estimated with the $N$ samples; the true distance is indicated by the circles. Notice that subspaces corresponding to the larger eigenvalue gap are much better approximated than the others. For example, the three-dimensional subspace is better approximated than the one- and two-dimensional subspaces for the third case.

\begin{figure}[ht]
\centering
\subfloat[Case 1]{
\includegraphics[width=0.3\linewidth]{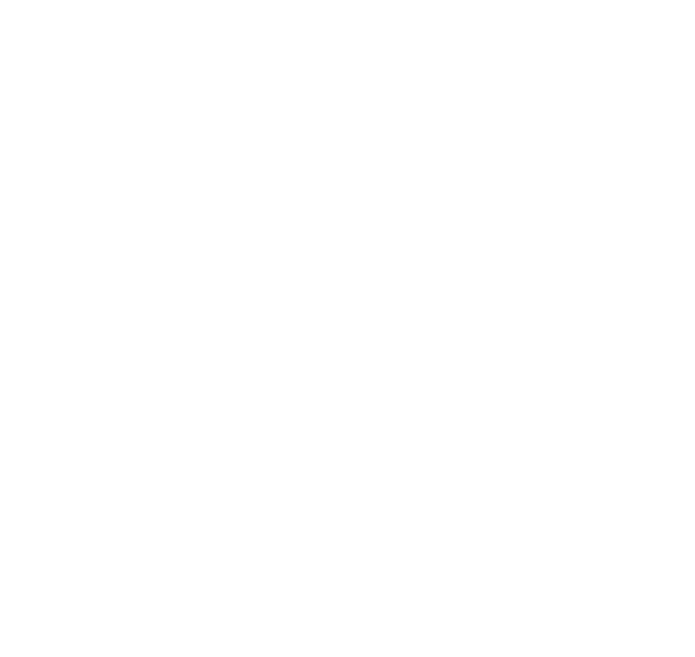}
}\;
\subfloat[Case 2]{
\includegraphics[width=0.3\linewidth]{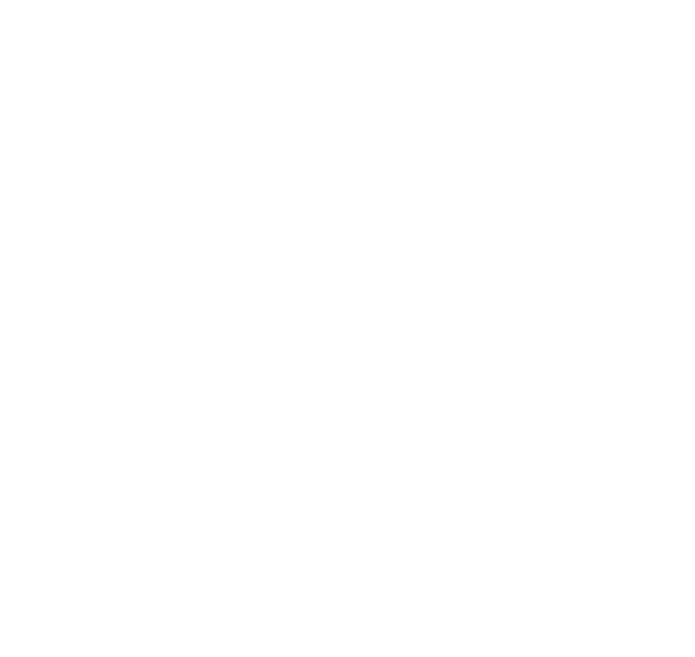}
}\;
\subfloat[Case 3]{
\includegraphics[width=0.3\linewidth]{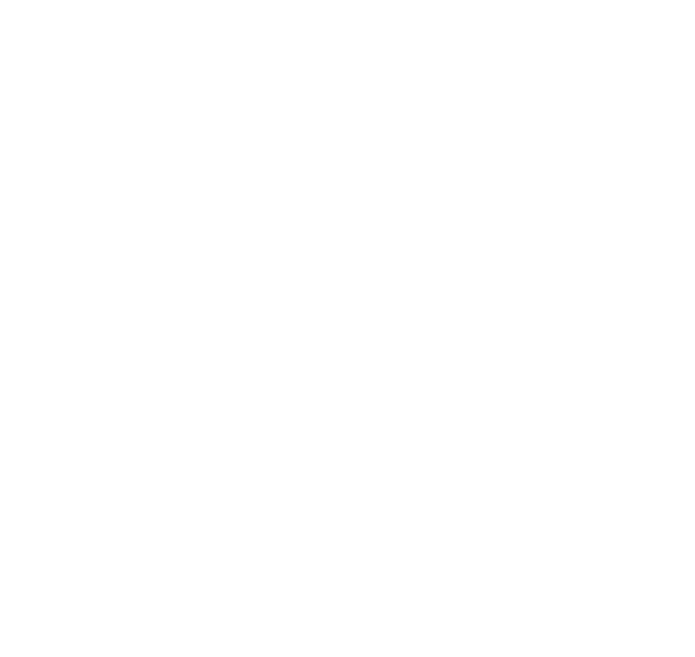}
}\\
\subfloat[Case 1]{
\includegraphics[width=0.3\linewidth]{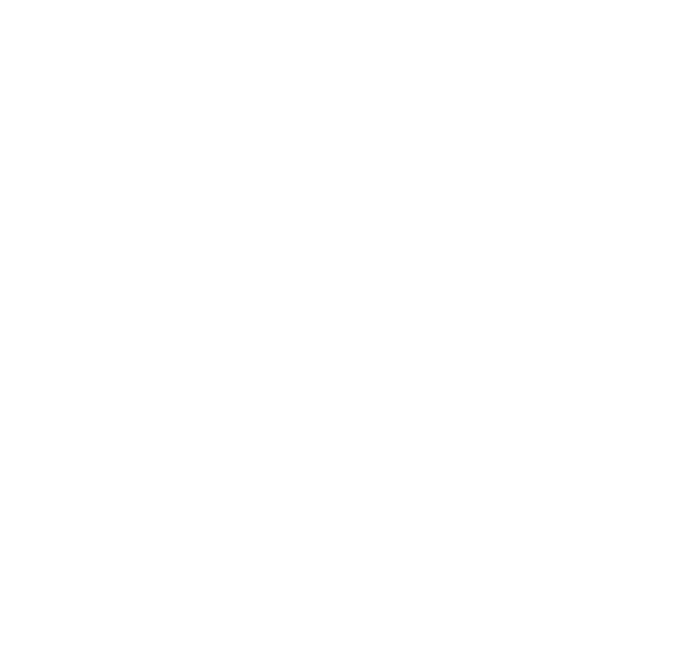}
}\:
\subfloat[Case 2]{
\includegraphics[width=0.3\linewidth]{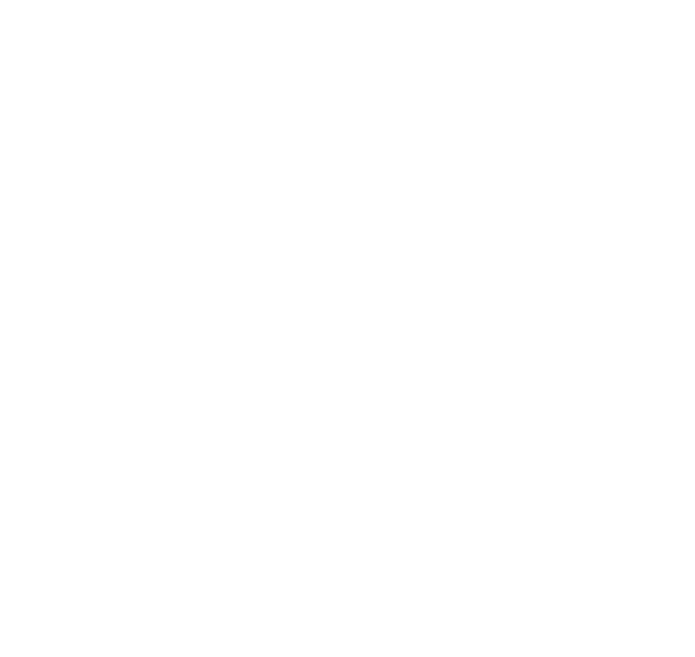}
}\;
\subfloat[Case 3]{
\includegraphics[width=0.3\linewidth]{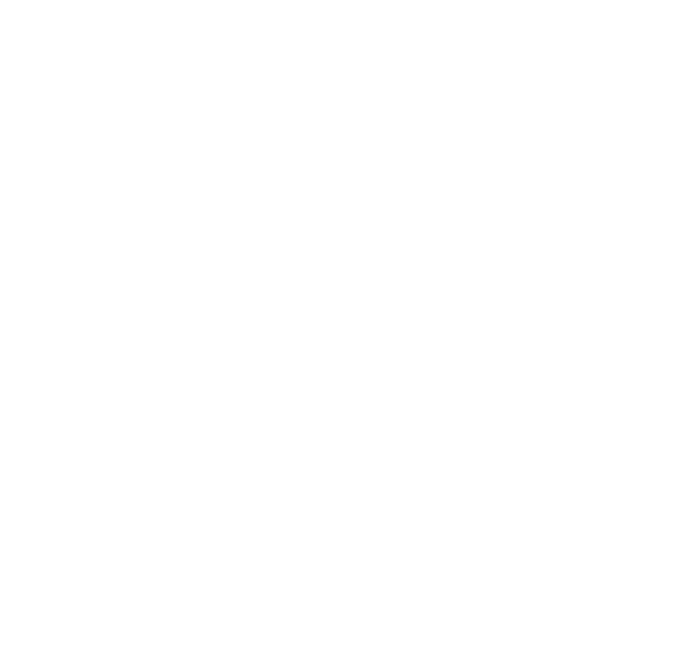}
}\\
\subfloat[Case 1]{
\includegraphics[width=0.3\linewidth]{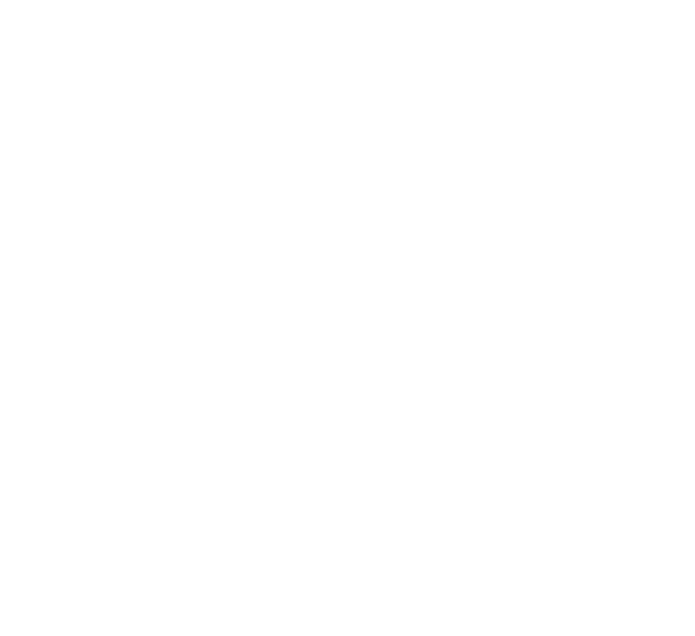}
}\;
\subfloat[Case 2]{
\includegraphics[width=0.3\linewidth]{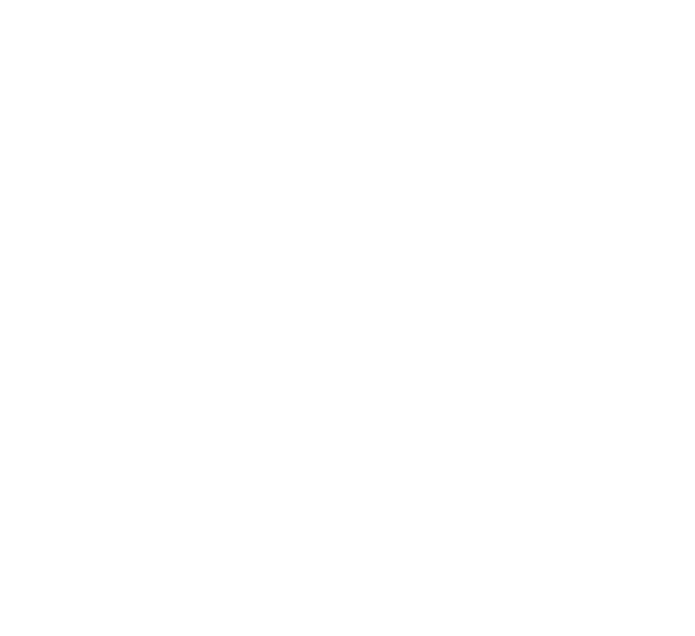}
}\;
\subfloat[Case 3]{
\includegraphics[width=0.3\linewidth]{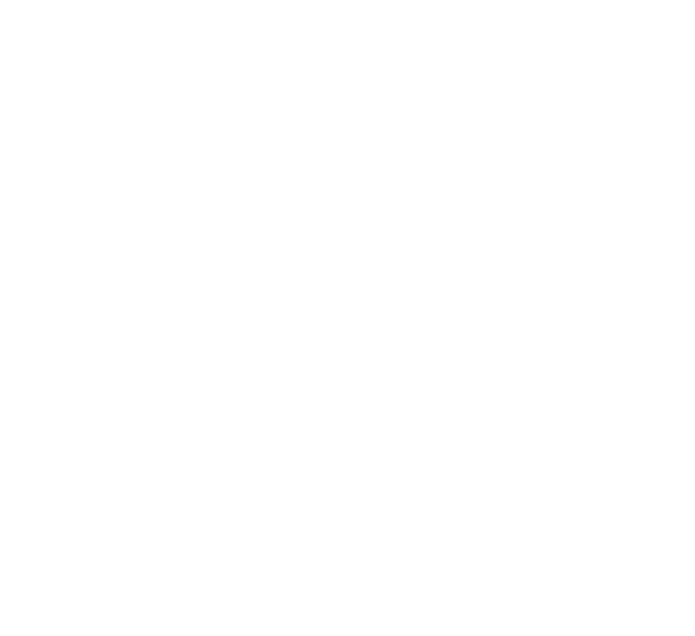}
}
\caption{The top row shows the eigenvalues of the three choices for $\mA$. The second row shows the true and estimated eigenvalues along with the bootstrap intervals; eigenvalues are well approximated for all three cases. The third row shows the distance between the estimated subspace and the true subspace. In practice we do not have the true subspace, but we can estimate the distance with a bootstrap procedure as described in Section \ref{sec:prac}; the bootstrap intervals are shown, and the accuracy of the subspace estimates corresponds to the gaps in the eigenvalues of $\mC$. }
\label{fig:quad0}
\end{figure} 

Next we repeat the study using finite difference approximations of the gradient with step size $h=10^{-1}$, $10^{-3}$, and $10^{-5}$. The first of these step sizes is larger than would normally be used for such a model. We chose this large value to study the interplay between inaccurate gradients and the finite sample approximations of the eigenpairs. Figure \ref{fig:quad1} shows the true eigenvalues, their estimates, and the bootstrap intervals for all three cases and all three values of $h$; the horizontal lines show the value of $h$. Eigenvalues that are smaller than $h$ are estimated less accurately than those larger than $h$, which is not surprising since we are using first order finite differences. Also the gaps are much less noticeable in the estimates when finite difference parameter is not small enough to resolve the smaller eigenvalue in the pair defining the gap. In fact, this particular problem shows a large gap in the finite difference approximations when there is none in the true eigenvalues; see Figure~\ref{fig:quad1}(b,c) for examples of this phenomenon. 

Figure~\ref{fig:quad2} shows the distance between the true active subspace and the finite sample estimate with approximate gradients (circles). We use the bootstrap to estimate the error in the subspace as in Section \ref{sec:prac}. There is a strong bias in the estimates of the subspace error when the corresponding eigenvalues are not properly resolved. For instance, in Figure~\ref{fig:quad2}, the estimates of the error for subspaces of dimension 4 through 6 are biased for $h=10^{-3}$ and significantly biased for $h=10^{-1}$. Compare this to the error in the last three eigenvalues for the smallest $h=10^{-5}$ in Figure \ref{fig:quad1}. 

\begin{figure}[ht]
\centering
\subfloat[Case 1, $h=10^{-1}$]{
\includegraphics[width=0.3\linewidth]{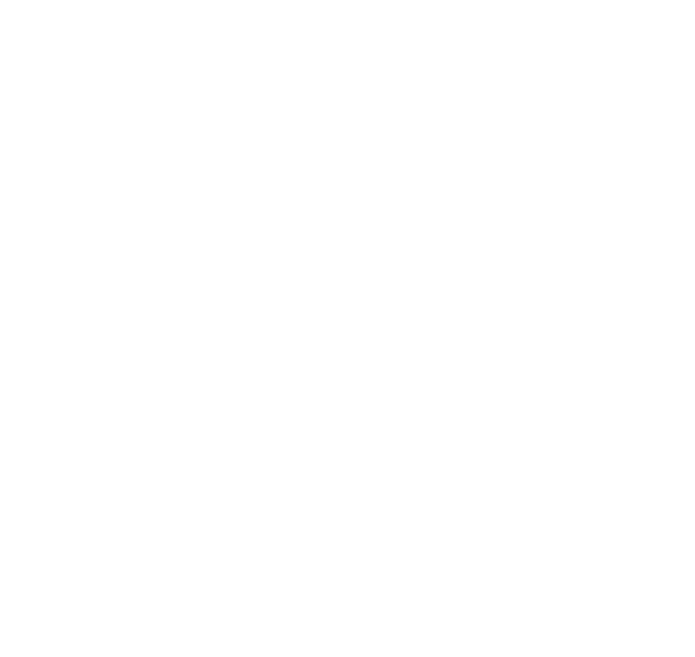}
}\;
\subfloat[Case 2, $h=10^{-1}$]{
\includegraphics[width=0.3\linewidth]{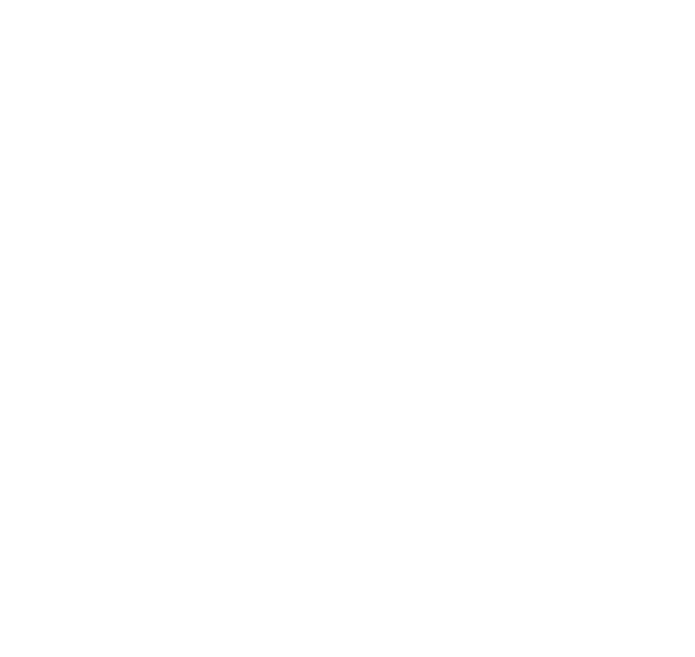}
}\;
\subfloat[Case 3, $h=10^{-1}$]{
\includegraphics[width=0.3\linewidth]{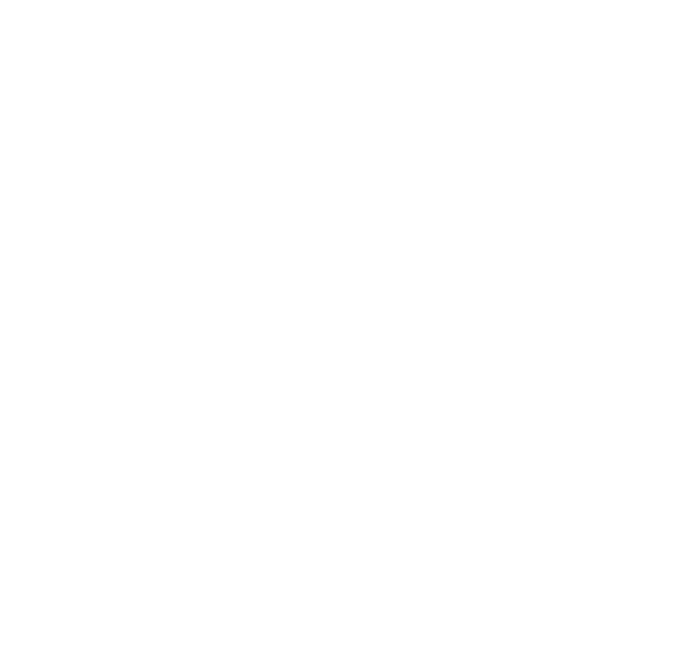}
}\\
\subfloat[Case 1, $h=10^{-3}$]{
\includegraphics[width=0.3\linewidth]{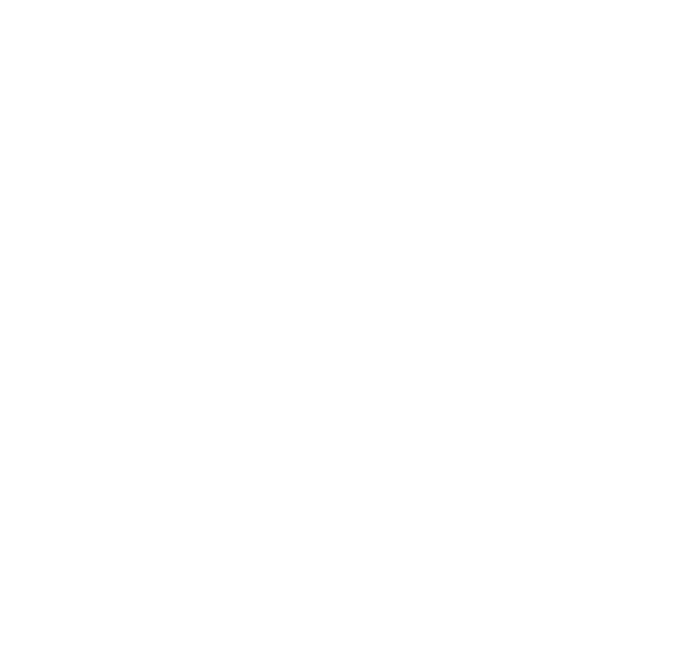}
}\;
\subfloat[Case 2, $h=10^{-3}$]{
\includegraphics[width=0.3\linewidth]{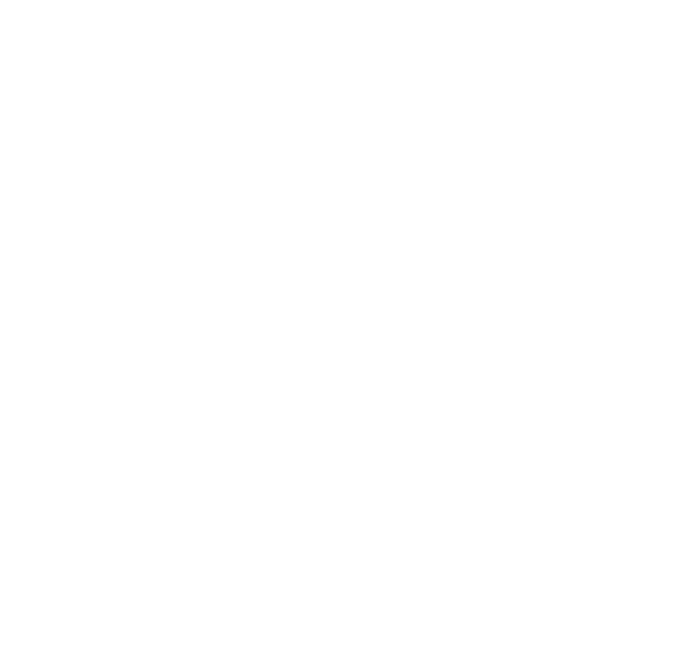}
}\;
\subfloat[Case 3, $h=10^{-3}$]{
\includegraphics[width=0.3\linewidth]{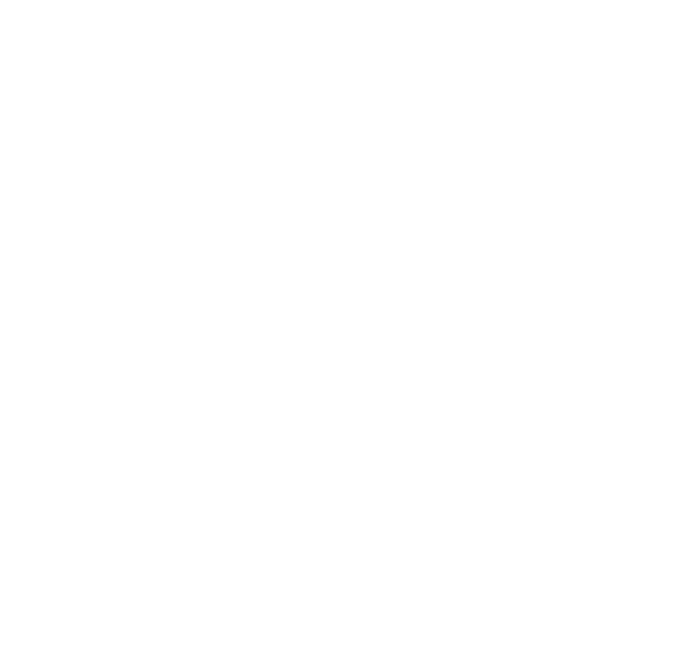}
}\\
\subfloat[Case 1, $h=10^{-5}$]{
\includegraphics[width=0.3\linewidth]{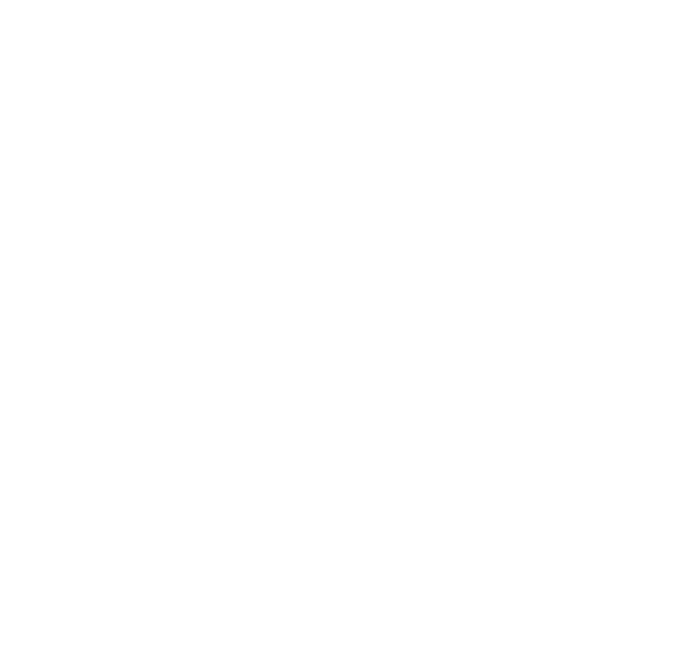}
}\;
\subfloat[Case 2, $h=10^{-5}$]{
\includegraphics[width=0.3\linewidth]{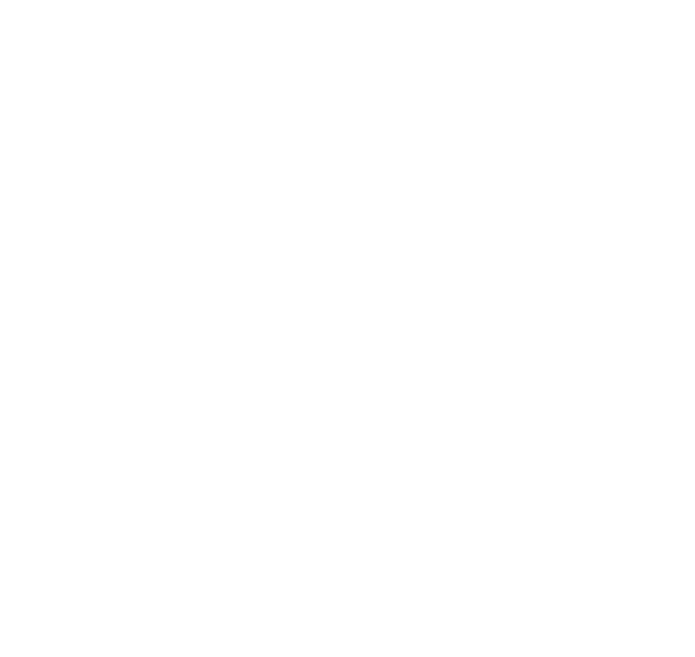}
}\;
\subfloat[Case 3, $h=10^{-5}$]{
\includegraphics[width=0.3\linewidth]{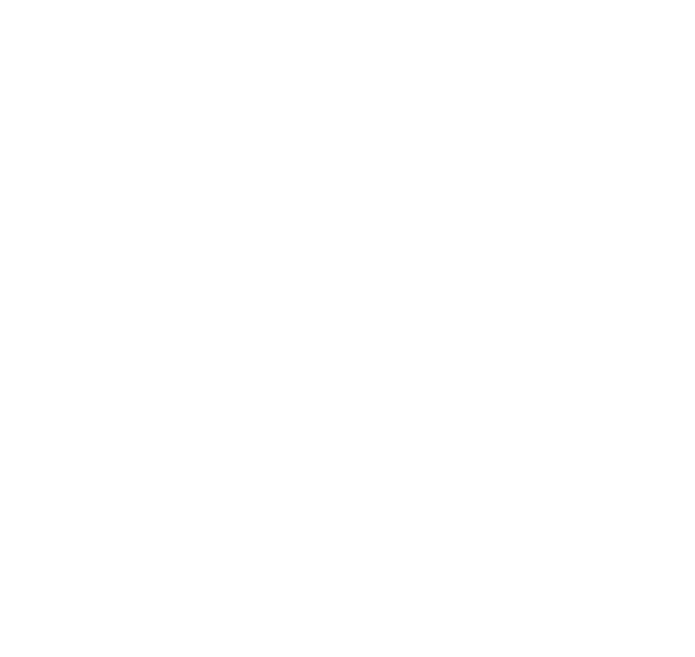}
}
\caption{Eigenvalues, estimates, and bootstrap intervals using finite difference gradients with $h=10^{-1}$ (top row), $h=10^{-3}$ (middle row), and $h=10^{-5}$ (bottom row). The horizontal black lines indicate the value of $h$ in each plot. In general, estimates of eigenvalues smaller than $h$ are less accurate than those larger than $h$.}
\label{fig:quad1}
\end{figure}

\begin{figure}[ht]
\centering
\subfloat[Case 1, $h=10^{-1}$]{
\includegraphics[width=0.3\linewidth]{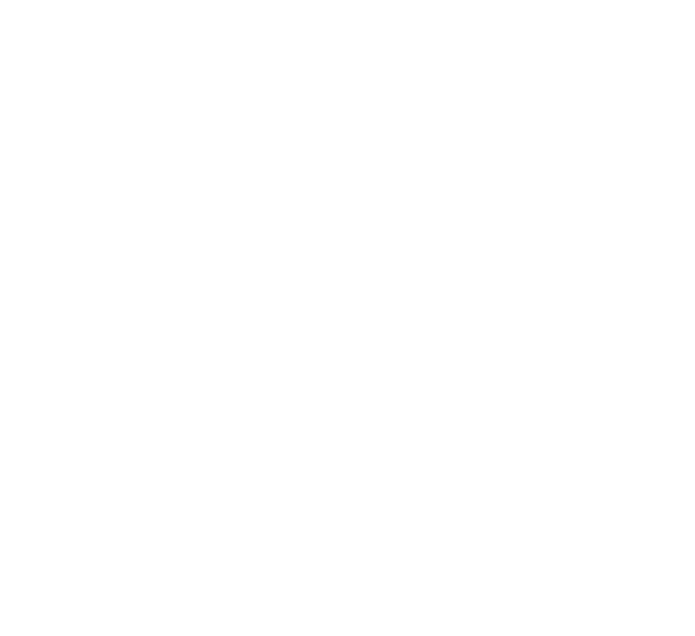}
}\;
\subfloat[Case 2, $h=10^{-1}$]{
\includegraphics[width=0.3\linewidth]{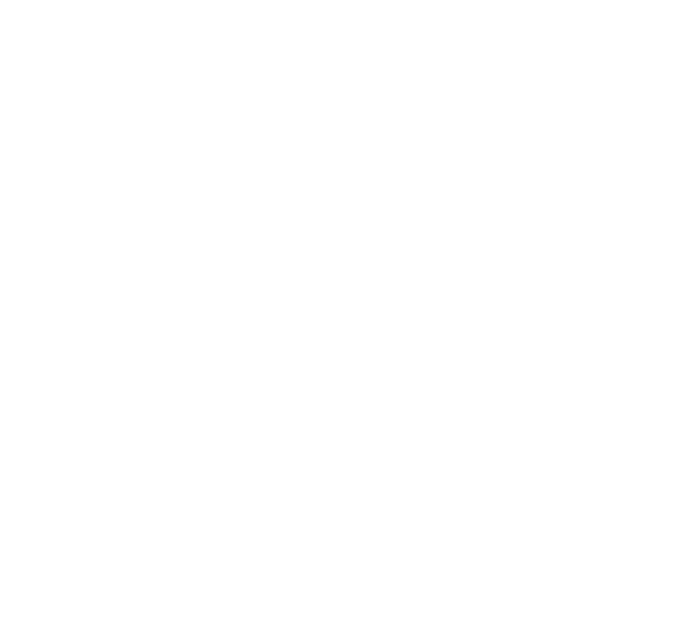}
}\;
\subfloat[Case 3, $h=10^{-1}$]{
\includegraphics[width=0.3\linewidth]{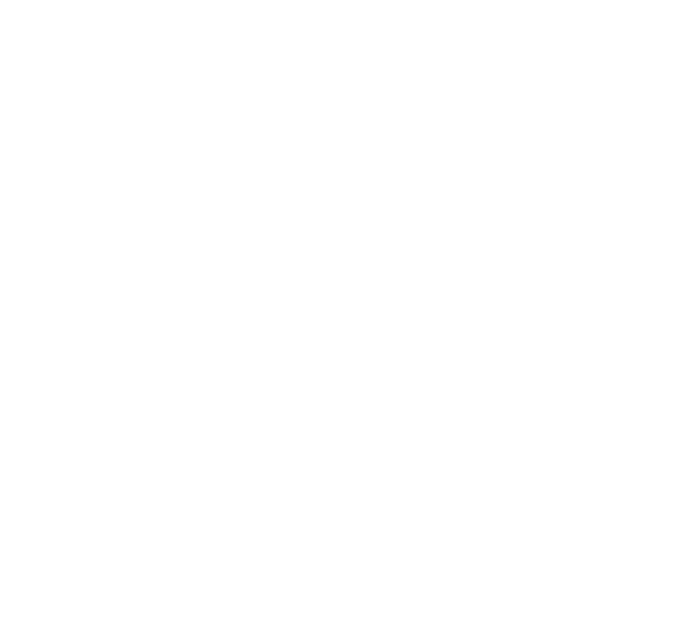}
}\\
\subfloat[Case 1, $h=10^{-3}$]{
\includegraphics[width=0.3\linewidth]{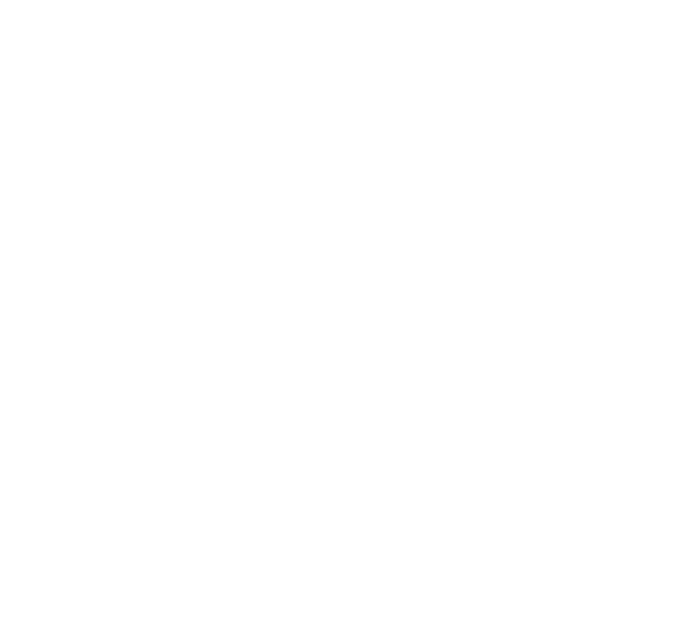}
}\;
\subfloat[Case 2, $h=10^{-3}$]{
\includegraphics[width=0.3\linewidth]{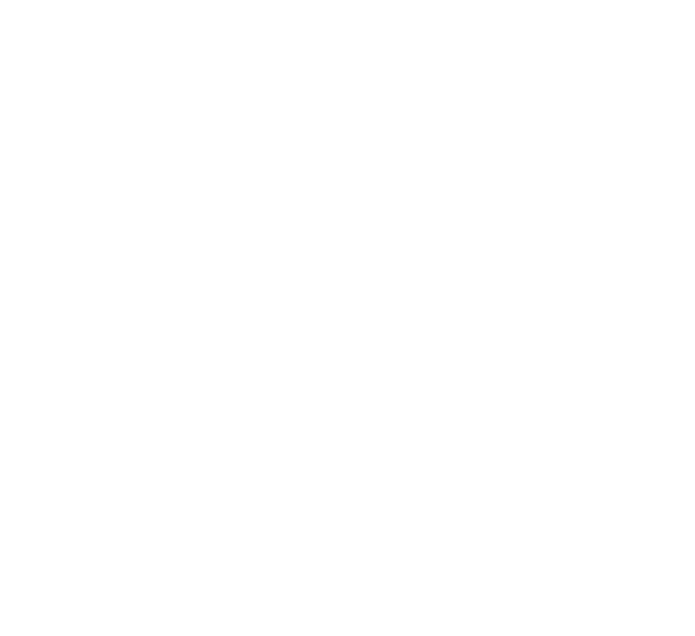}
}\;
\subfloat[Case 3, $h=10^{-3}$]{
\includegraphics[width=0.3\linewidth]{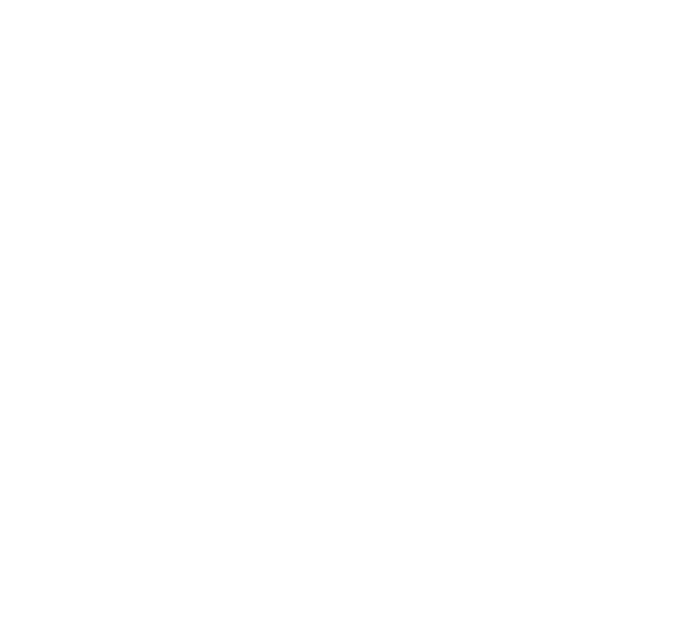}
}\\
\subfloat[Case 1, $h=10^{-5}$]{
\includegraphics[width=0.3\linewidth]{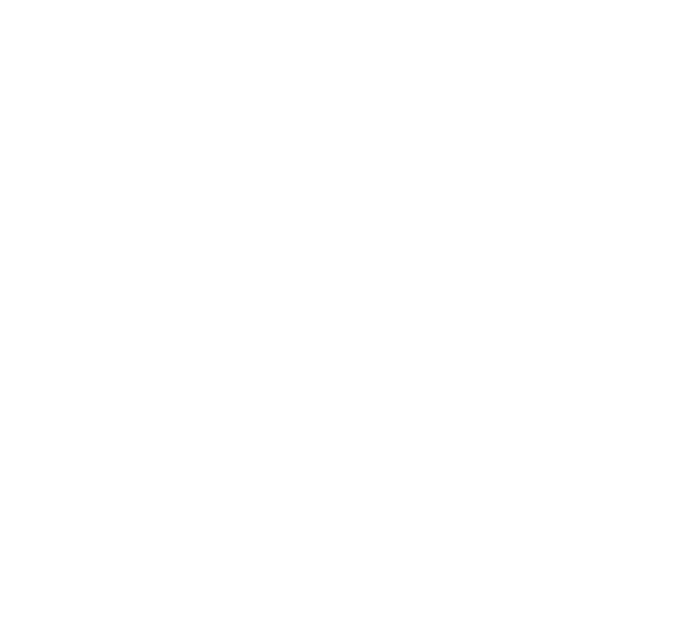}
}\;
\subfloat[Case 2, $h=10^{-5}$]{
\includegraphics[width=0.3\linewidth]{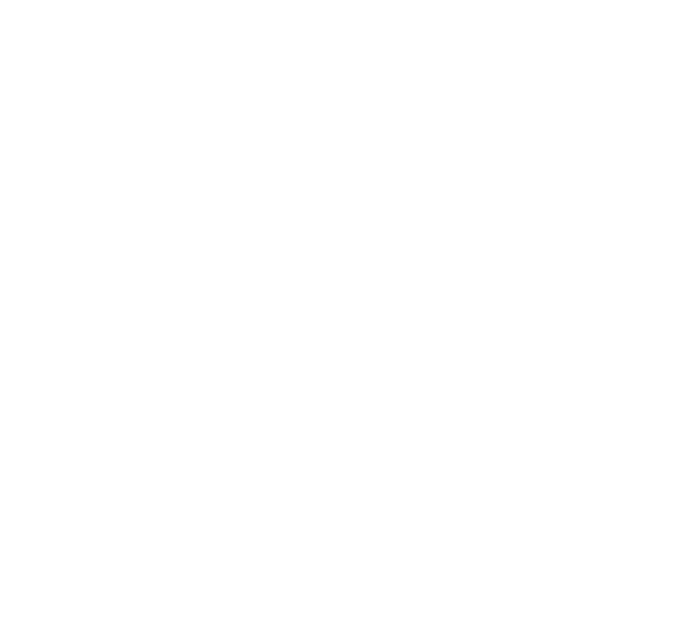}
}\;
\subfloat[Case 3, $h=10^{-5}$]{
\includegraphics[width=0.3\linewidth]{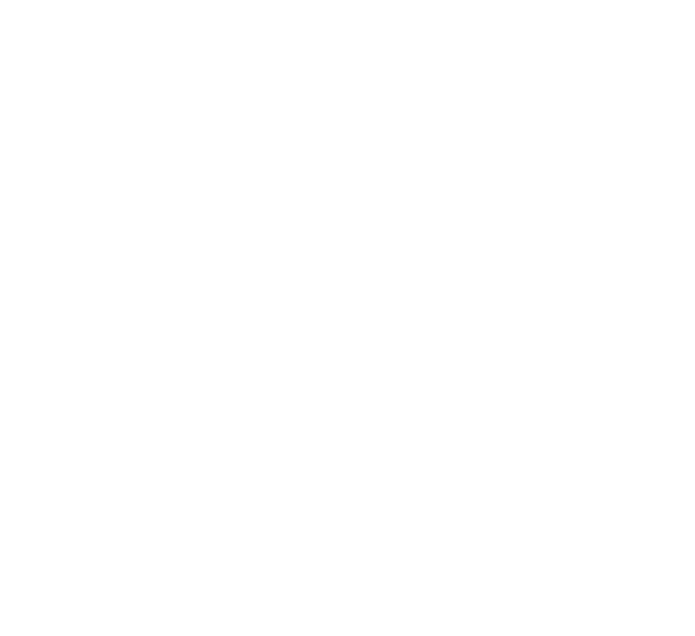}
}
\caption{The distance between the true active subspace and its finite sample approximation along with bootstrap intervals for $h=10^{-1}$ (top row), $h=10^{-3}$ (middle row), and $h=10^{-5}$ (bottom row). The subspaces are very poorly approximated when the finite difference step size is not small enough to resolve the eigenvalues corresponding the subspaces; compare to Figure \ref{fig:quad1}. However, subspaces with a larger associated eigenvalue gap are generally approximated better than others. }
\label{fig:quad2}
\end{figure} 

\subsection{A parameterized PDE model}
In previous work~\cite{constantine2014active}, we exploited the active subspace in the following parameterized PDE model to efficiently construct a kriging surface. Here we perform a more careful study of the variation in the active subspace estimated with finite samples of the gradient. Consider the following linear elliptic PDE with parameterized, variable coefficients. Let $u=u(\vs,\vx)$ satisfy
\begin{equation}
-\nabla_\vs\cdot(a\,\nabla_\vs u) \;=\; 1, \qquad \vs\in[0,1]^2.
\end{equation}
We set homogeneous Dirichlet boundary conditions on the left, top, and bottom of the spatial domain $[0,1]^2$; denote this boundary by $\Gamma_1$. The right side of the spatial domain, denoted $\Gamma_2$, has a homogeneous Neumann boundary condition. The log of the coefficients $a=a(\vs,\vx)$ is given by a truncated Karhunen-Loeve-type expansion
\begin{equation}
\label{eq:kl}
\log(a(\vs,\vx)) \;=\; \sum_{i=1}^m x_i\,\gamma_i\,\phi_i(\vs),
\end{equation}
where the $x_i$ are independent, identically distributed standard normal random variables, and the $\{\phi_i(\vs),\gamma_i\}$ are the eigenpairs of the correlation operator
\begin{equation}
\label{eq:corr}
\mathcal{C}(\vs,\vt) \;=\; \exp\left(-\beta^{-1}\,\|\vs-\vt\|_1\right).
\end{equation}
We study the quality of the active subspace approximation for two correlation lengths, $\beta=1$ and $\beta=0.01$. These correspond to \emph{long} and \emph{short} correlation lengths, respectively, for the random field defining the log of the coefficients. We choose a truncation of the field $m=100$, which implies that the parameter space $\sX=\mathbb{R}^{100}$ with $\rho$ a standard Gaussian density function. Define the linear function of the solution 
\begin{equation}
\label{eq:qoi}
f(\vx) \;=\; \frac{1}{|\Gamma_2|} \int_{\Gamma_2} u(\vs,\vx) \,d\vs.
\end{equation}
This is the quantity of interest from the model (more precisely, its approximation with a finite element method). Given a value for the input parameters $\vx$, we discretize the PDE with a standard linear finite element method using MATLAB's PDE Toolbox. The discretized domain has 34320 triangles and 17361 nodes; the eigenfunctions $\phi_i$ from \eqref{eq:kl} are approximated on this mesh. We compute the gradient of the quantity of interest \eqref{eq:qoi} using a discrete adjoint formulation. Further details appear in our previous work~\cite{constantine2014active}.

The top row of Figure \ref{fig:pde0} shows the estimates of the eigenvalues of $\mC$ along with the bootstrap intervals for $\beta=1$ in \eqref{eq:corr}. The gap between the first and second eigenvalues is apparent and supported by the gap in the corresponding bootstrap intervals. We exploit this gap in~\cite{constantine2014active} to construct an accurate univariate kriging surface of the active variable. The bottom row of Figure \ref{fig:pde0} shows the variance in the estimated subspace as computed with the bootstrap including the bootstrap intervals. The left column of Figure \ref{fig:pde0} uses the multiplier $\alpha=2$ when choosing the number $N$ of gradient samples; the right column uses $\alpha=10$. Notice the overall decrease in both the range of the bootstrap interval and the subspace error as we include more samples. Figure \ref{fig:pde1} shows the identical study with the short correlation length $\beta=0.01$ from \eqref{eq:corr}. 

\begin{figure}[ht]
\centering
\subfloat[$\beta=1$, $\alpha=2$]{
\includegraphics[width=0.45\linewidth]{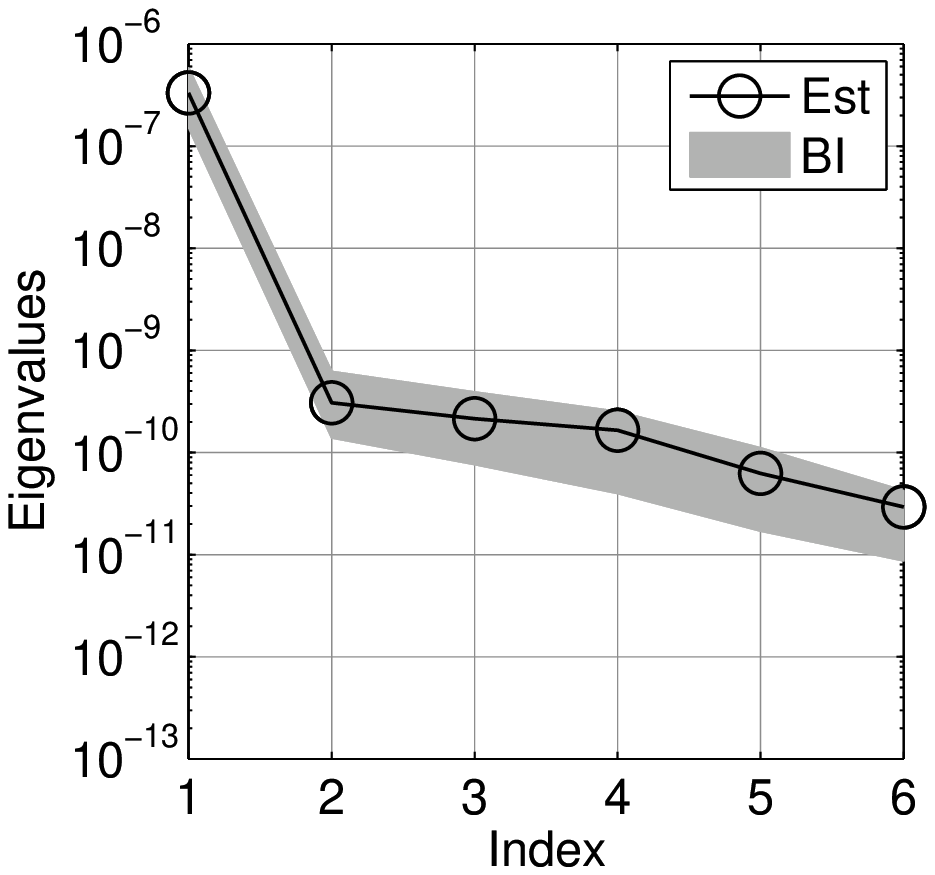}
}\;
\subfloat[$\beta=1$, $\alpha=10$]{
\includegraphics[width=0.45\linewidth]{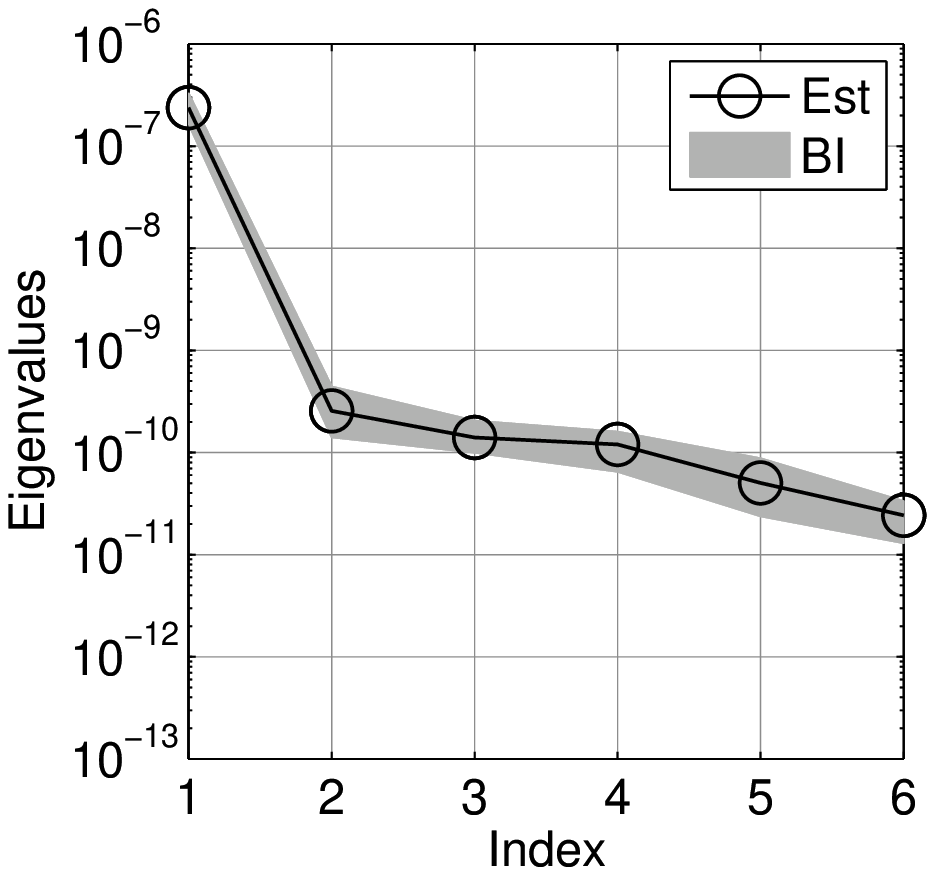}
}\\
\subfloat[$\beta=1$, $\alpha=2$]{
\includegraphics[width=0.45\linewidth]{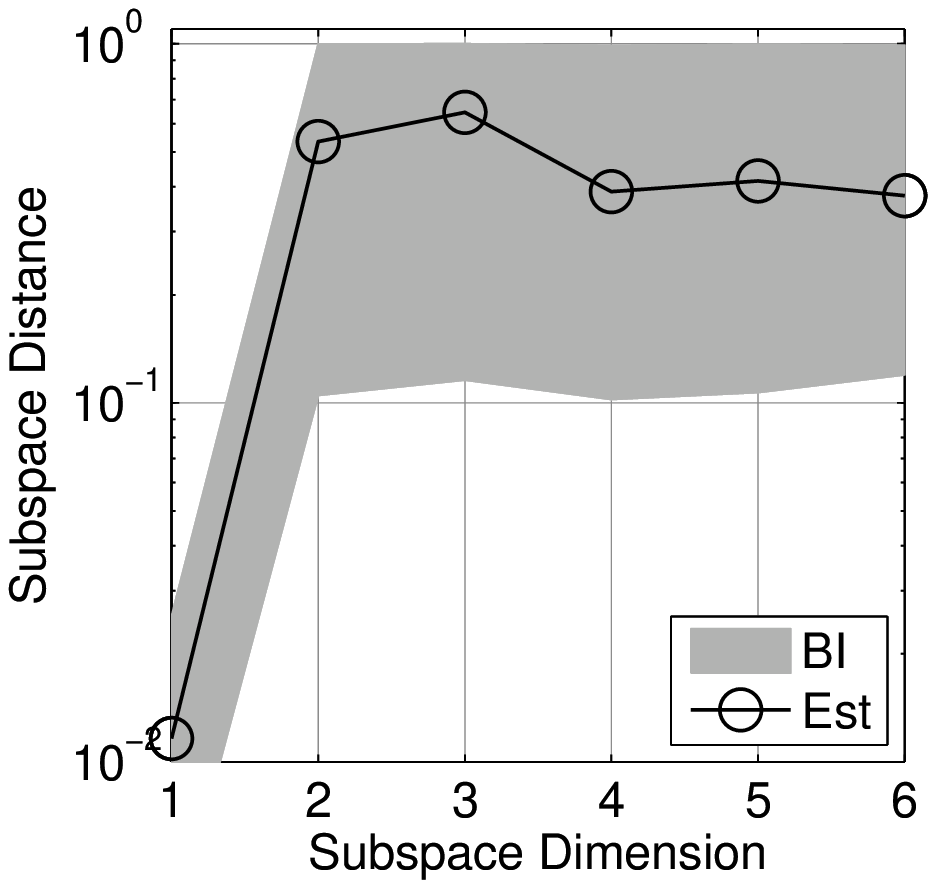}
}\;
\subfloat[$\beta=1$, $\alpha=10$]{
\includegraphics[width=0.45\linewidth]{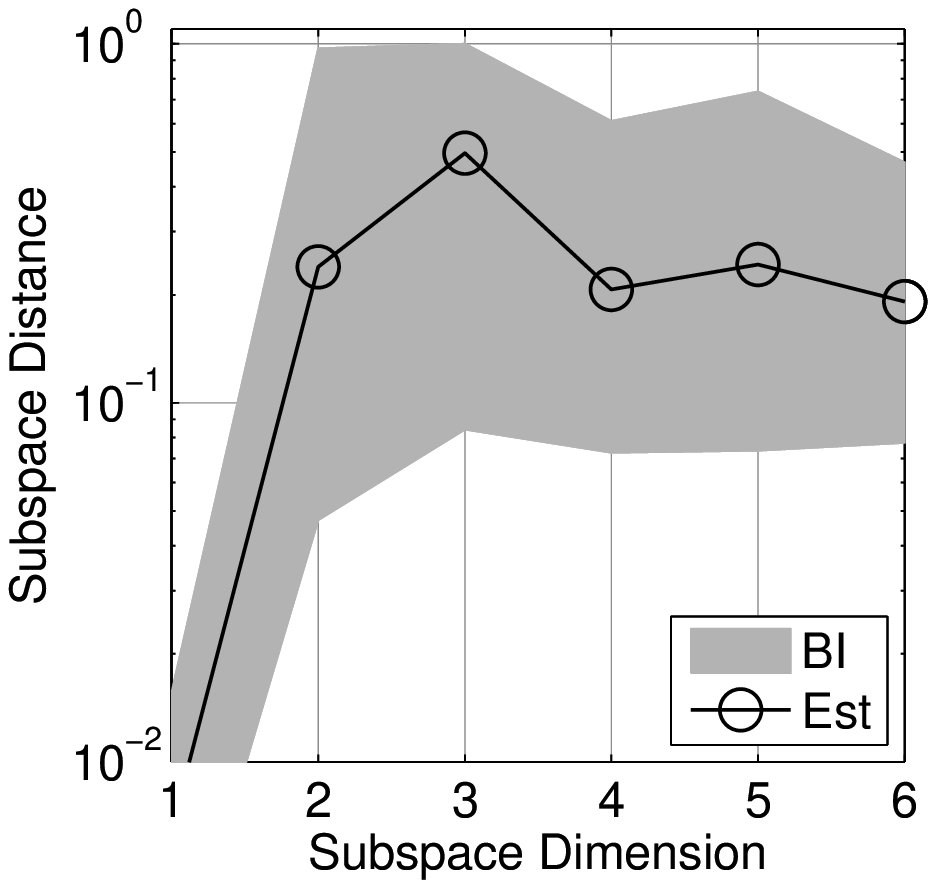}
}
\caption{The top row shows estimates of the eigenvalues of $\mC$ along with the bootstrap intervals for the quantity of interest \eqref{eq:qoi} from the parameterized PDE model with the long correlation length $\beta=1$ from \eqref{eq:corr}. The bottom row shows the estimates and bootstrap intervals on the distance between the estimated active subspace and the true active subspace. The left column is computed with the multiplier $\alpha=2$ when choosing $N$; the right column uses $\alpha=10$. The gap between the first and second eigenvalue is significant as judged by the gap between the bootstrap intervals.}
\label{fig:pde0}
\end{figure} 

\begin{figure}[ht]
\centering
\subfloat[$\beta=0.01$, $\alpha=2$]{
\includegraphics[width=0.45\linewidth]{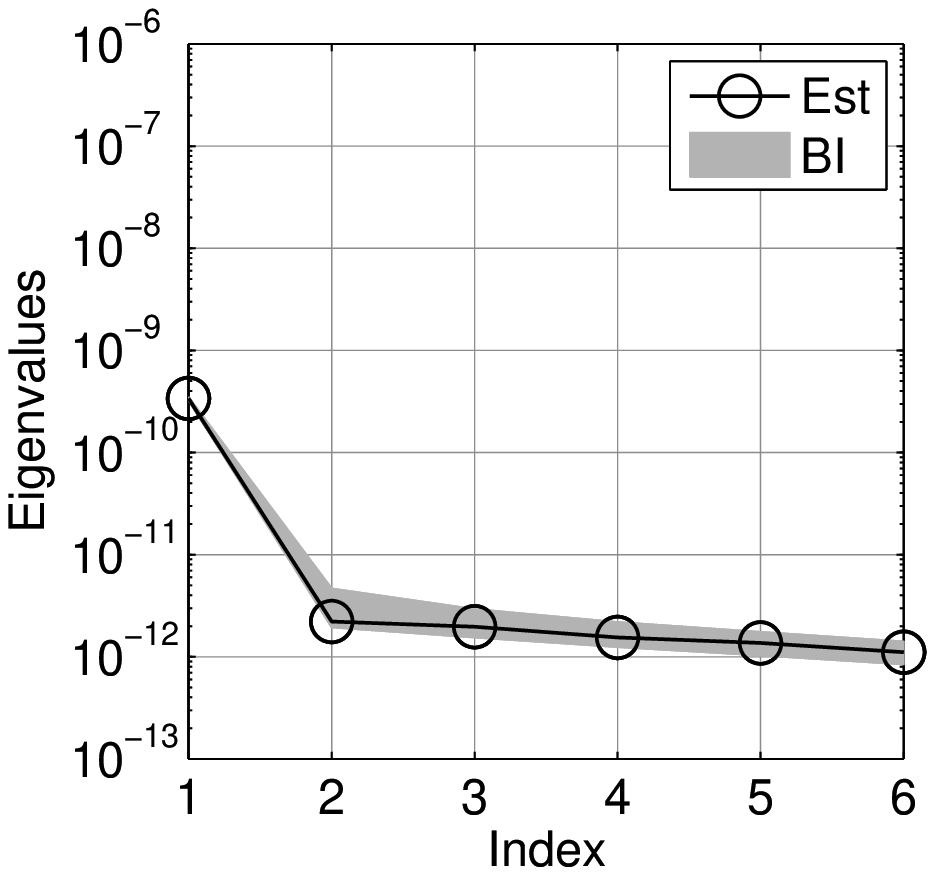}
}\;
\subfloat[$\beta=0.01$, $\alpha=10$]{
\includegraphics[width=0.45\linewidth]{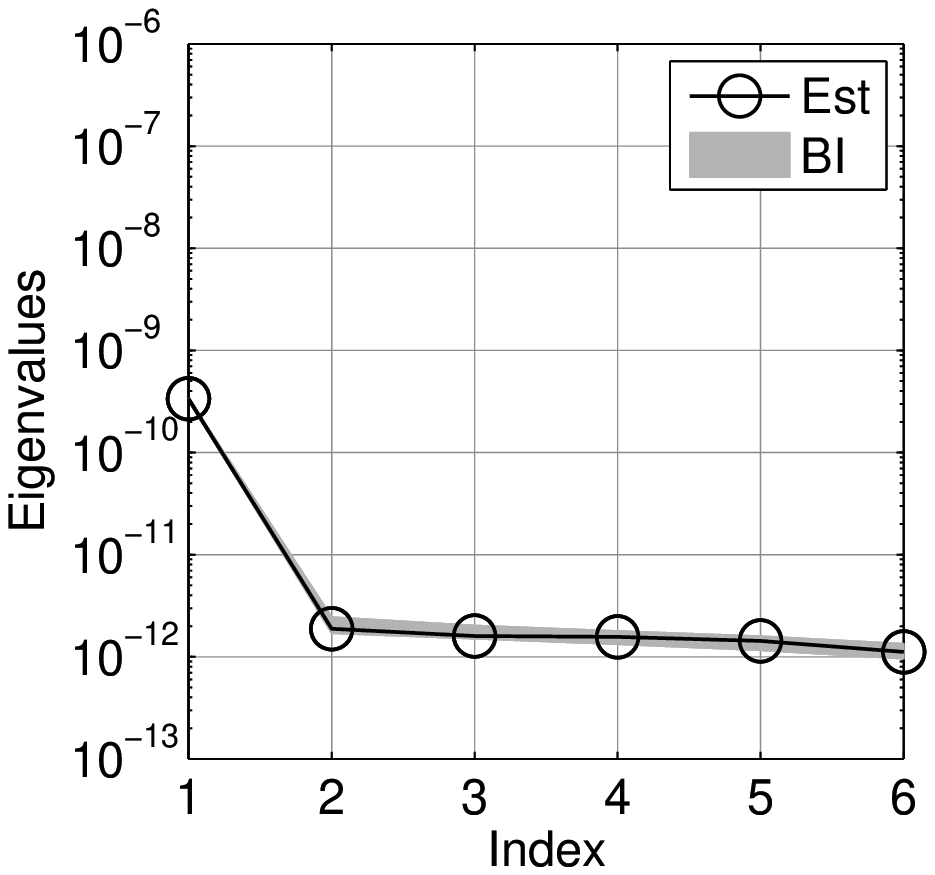}
}\\
\subfloat[$\beta=0.01$, $\alpha=2$]{
\includegraphics[width=0.45\linewidth]{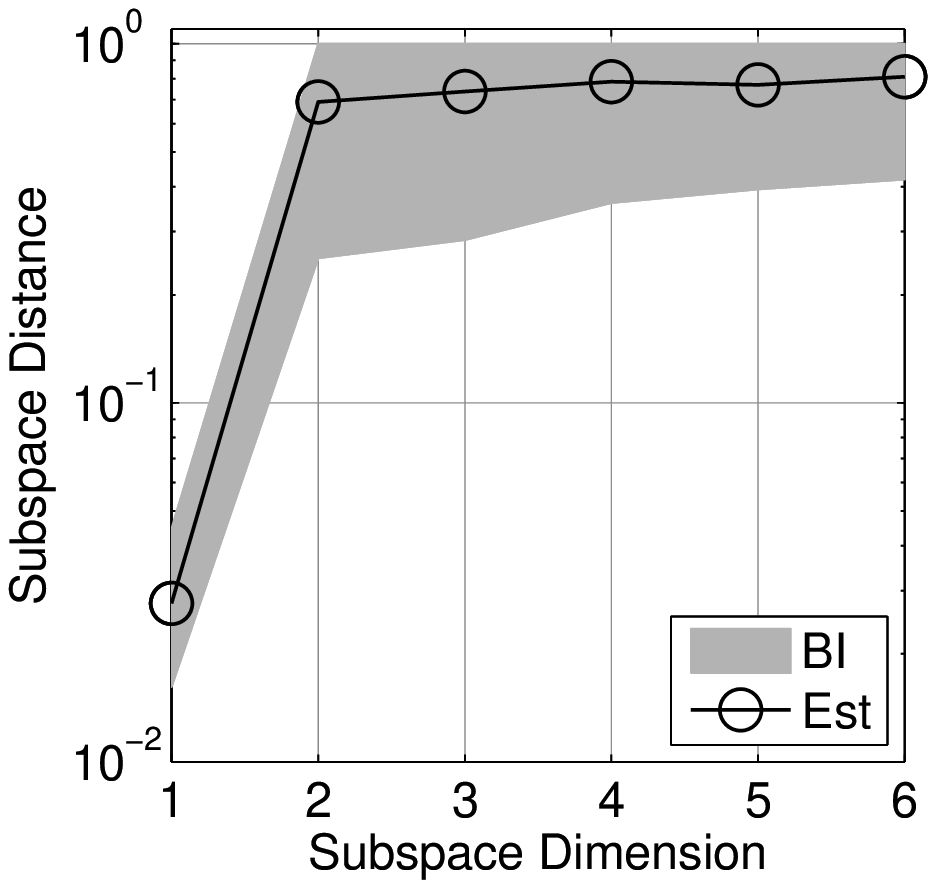}
}\;
\subfloat[$\beta=0.01$, $\alpha=10$]{
\includegraphics[width=0.45\linewidth]{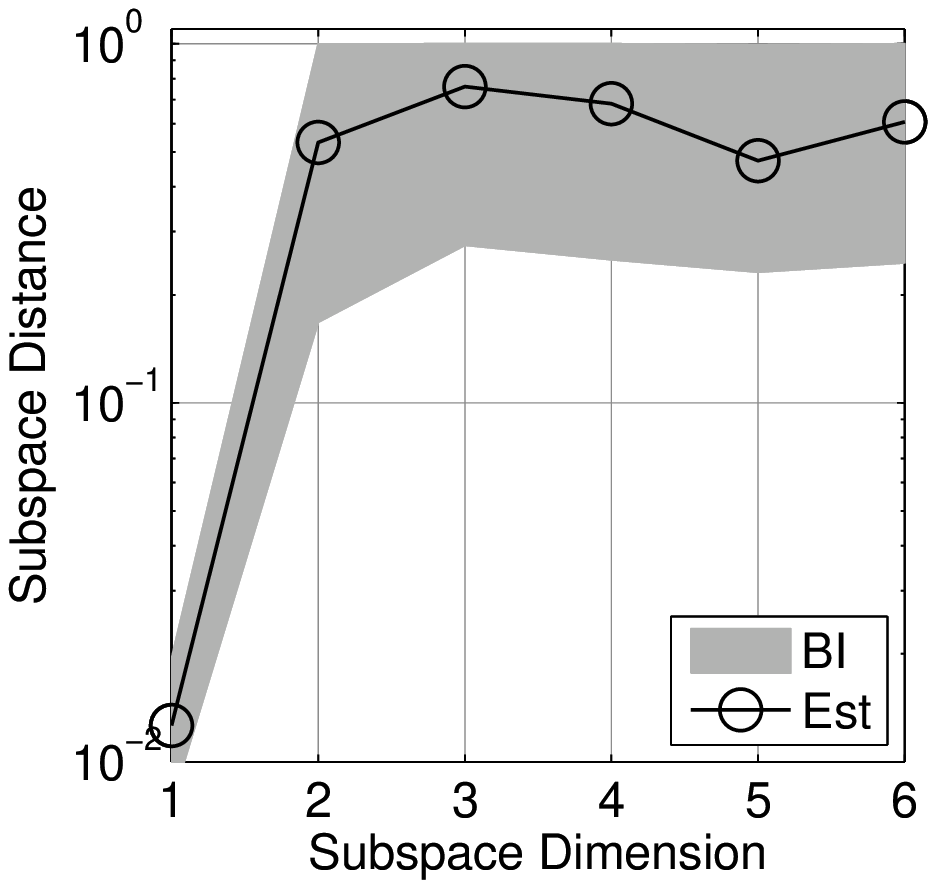}
}
\caption{The top row shows estimates of the eigenvalues of $\mC$ along with the bootstrap intervals for the quantity of interest \eqref{eq:qoi} from the parameterized PDE model with the short correlation length $\beta=0.01$ from \eqref{eq:corr}. The bottom row shows the estimates and bootstrap intervals on the distance between the estimated active subspace and the true active subspace. The left column is computed with the multiplier $\alpha=2$ when choosing $N$; the right column uses $\alpha=10$. The gap between the first and second eigenvalue is significant as judged by the gap between the bootstrap intervals.}
\label{fig:pde1}
\end{figure} 

\section{Summary and conclusions}

Consider a scalar-valued function of several variables. The average outer product of the gradient with itself is the central matrix in the development of active subspaces for dimension reduction. The dominant eigenvectors define the directions along which input perturbations change the output more, on average. We have analyzed a Monte Carlo method for approximating this matrix and its eigenpairs. We use recent theory developed for the eigenvalues of sums of random matrices to analyze the probability that the finite sample eigenvalue estimates deviate from the true eigenvalues, and we combine this analysis with results from matrix computations to derive results for the subspaces. We extend this analysis to quantities computed with samples of approximate gradients, e.g., finite differences. We also provide a practical computational approach that employs the bootstrap to reveal the error in the eigenvalues and the stability of the subspace. 

Our analysis offers answers to the following important questions. First, how many gradient samples does one need for an accurate approximation of the first $k$ eigenvalues? Precise theoretical bounds motivate a heuristic that chooses a number proportional to $k$ times the log of the dimension $m$. Second, what can be said about the accuracy of the estimated subspace? The accuracy of the estimated subspace is directly related to gaps in the eigenvalues. Third, how does one judge the stability of the computed quantities? We propose to use bootstrap intervals for the eigenvalues and the stability of the subspace. Finally, how does this analysis change when gradients are not exact but approximate? Our theory shows that approximate gradients introduce a bias term in the error bounds that goes to zero as the approximate gradients become more accurate. The numerical examples suggest that this bias can produce inaccurate subspaces when the gradients are not well approximated. 

\section*{Acknowledgments}
The first author was partially supported by the U.S. Department of Energy Office of Science, Office of Advanced Scientific Computing Research, Applied Mathematics program under Award Number DE-SC-0011077. 
The second author would like to thank the Simon's Institute for Theory of Computing program on Big Data for the opportunity to learn about the
randomized methods used in this paper and NSF CAREER award CCF-114975.

\bibliographystyle{siam}
\bibliography{compute-asm}

\end{document}